\documentclass[12pt]{amsart}
\usepackage{amsmath}
\usepackage{extarrows}
\usepackage{amsfonts}
\usepackage{amssymb}
\usepackage[all,cmtip]{xy}           

\usepackage{bbm}
\usepackage{bbding}
\usepackage{txfonts}
\usepackage[shortlabels]{enumitem}
\usepackage{ifpdf}
\ifpdf
\usepackage[colorlinks,final,backref=page,hyperindex]{hyperref}
\else
\usepackage[colorlinks,final,backref=page,hyperindex,hypertex]{hyperref}
\fi
\usepackage{tikz-cd}
\usepackage[active]{srcltx}
\usepackage{tikz}
\usepackage{amscd}
\usepackage{bm}
\usepackage{mathrsfs}

\makeatletter

\newtheorem{thm}{Theorem}[section]
\newtheorem{prop}[thm]{Proposition}
\newtheorem{lem}[thm]{Lemma}
\newtheorem{cor}[thm]{Corollary}
\newtheorem{prop-def}{Proposition-Definition}[section]

\newtheorem{defn}[thm]{Definition}

\newtheorem{rmk}[thm]{Remark}
\newtheorem{exam}[thm]{Example}

\newcommand{\nc}{\newcommand}
\nc{\Sh}{{\mathrm{Sh}}}
\nc{\ot}{\otimes}
\nc{\Id}{\mathrm{Id}}
\nc{\id}{\mathrm{Id}}
\nc{\NR}{{\rm NR}}
\nc{\G}{{\rm G}}
\nc{\rmS}{{\rm S}}
\nc{\T}{{\rm T}}
\nc{\Hom}{\mathrm{Hom}}
\nc{\rmC}{ {\mathrm{C}}}
\nc{\Alg}{\mathsf{Alg}}
\nc{\RB}{{\mathsf{RB}}}
\nc{\ad}{\mathrm{ad}}
\nc{\Lie}{\mathrm{Lie}}
\nc{\RBA}{{\mathrm{RBA}_{(\mu,\lambda)}}}
\nc{\C}{\mathrm{C}}
\nc{\rmH}{\mathrm{H}}
\nc{\rar}{\rightarrow}
\nc{\End}{\mathrm{End}}

\def\bc{\begin{center}}
	\def\ec{\end{center}}

\def\hang{\hangindent\parindent}
\def\textindent#1{\indent\llap{\qquad #1\ \ \enspace}\ignorespaces}
\def\ref{\par\hang\textindent}

\def\o{\otimes}

\def \frakg{ \frak g}
\def\bfk{\mathbf{k}}

\def\a{\phi_\mathfrak{g}}
\def \frakg{\mathfrak{g}}

\def \frakm{\mathfrak{m}}
\def \fraka{\mathfrak{a}}
\def \frakh{\mathfrak{h}}
\allowdisplaybreaks
\begin{document}
	\title[Extended Rota-Baxter algebras]{Deformation theory and the  controlling $L_{\infty}$-structure  of extended Rota-Baxter  algebras}

	\author{Jian Yang}
	\address{Jian Yang\\
		School of Mathematical Sciences, 
		East China Normal University,
		Shanghai 200241,
		China}

	\email{y.j0@qq.com }
	
	\date{\today}

		\begin{abstract} In this paper, we introduce the controlling $L_{\infty}[1]$-algebra of extended Rota-Baxter algebras. As applications, we obtain the cohomology of  extended Rota-Baxter  algebras and  homotopy extended Rota-Baxter  algebras.  We also study  abelian extensions and formal deformations by the lower cohomology group.  Lastly, we describe the intrinsic relationship among extended Rota-Baxter algebras, extended Rota-Baxter Lie algebras, and the eigenvalues of the extended Rota-Baxter operator.
			
			 As  the generalization of Rota-Baxter algebras with weight and modified Rota-Baxter algebras, the corresponding results also hold for those algebras. 
		\end{abstract}
	
	
	 
	\maketitle
	
	\tableofcontents
	
	\allowdisplaybreaks

	\section{Introduction}
	
	In this paper, we study  extended Rota-Baxter algebras, a generalization of Rota-Baxter algebras with weight and modified Rota-Baxter algebras. 
	
	\subsection{Extended Rota-Baxter algebras}\
	
	Baxter introduced Rota-Baxter algebras  in \cite{Bax60},  interestingly, this work originated from  research in probability theory. Then Rota \cite{Rot69}, Cartier \cite{Car72} and other researchers subsequently entered this field.   Semenov-Tian-Shansky independently discovered that the Rota-Baxter operator on  Lie algebras  is a solution of the classical Yang-Baxter equation\cite{Sem83}. Similarly, for the case of associative algebras, it was given by Aguiar \cite{Agu00} and Bai \cite{Bai07}. Subsequently, Connes and Kreimer  \cite{CK00} established a significant connection between Rota-Baxter algebras and mathematical physics  in the renormalization of quantum field theory. Moreover,  Guo and Keigher \cite{GK00a, GK00b} realized free commutative Rota-Baxter algebras. In \cite{Sem83}, the authors also introduced a modified version of the classical Yang-Baxter
	equation  whose solutions are just the modified Rota-Baxter operators, which they call  modified r-matrices.

	Rota-Baxter algebras yield a wealth of applications and linkages across different mathematical domains such as combinatorics \cite{Rot95}, multiple zeta values in number theory \cite{GZ08}, operad theory \cite{Agu01, BBGN13}, and Hopf algebras \cite{CK00}. The modified classical Yang-Baxter
	equation has important applications in the study of Lax equations, affine geometry on Lie groups, factorization problems
	in Lie algebras and compatible Poisson structures \cite{BGN10,B90,L99,S15}.
	
	Due to the importance of Rota-Baxter algebras and
	modified Rota-Baxter algebras, the notion of extended Rota-Baxter algebras is introduced by Zhang, Guo and Qiu \cite{ZGQ24}.  
	
   	Fix $\mu,\lambda\in \bfk$ where $\bfk$ is a field. An extended Rota-Baxter operator of weight $(\mu,\lambda)$ on an associative algebra $A$   is a linear operator $T:A\to A$ satisfying the identity
   
   \begin{equation} \label{eq:diffwt}
   	T(u)T(v)=T(T(u)v+uT(v)+\mu uv)+\lambda uv, \quad u, v\in A.
   \end{equation}
   An associative algebra endowed with an extended Rota-Baxter operator of weight $(\mu,\lambda)$ is called an extended Rota-Baxter algebra of weight $(\mu,\lambda)$.

    \subsection{Deformations and homotopy theory}\
    
    A central philosophy in deformation theory, inspired by the foundational work of Gerstenhaber, Nijenhuis, Richardson, Deligne and others, is that the deformation of a given mathematical structure can be described by a differential graded (dg) Lie algebra or, more generally, an $L_\infty$-algebra. Consequently, an essential problem in deformation theory is the explicit construction of the dg Lie algebra or $L_\infty$-algebra that governs the deformation theory of the structure under study.
    
    Another significant challenge in the study of algebraic structures is understanding their homotopy analogs, such as $A_\infty$-algebras for associative algebras and $L_\infty$-algebras for Lie algebras.

    	The deformation theory  and  homotopy theory of  Rota-Baxter  algebras had been absent for a long time despite   the   importance of Rota-Baxter algebras.  Recently there are some breakthroughs in this direction.

     Fortunately, there exists a powerful technique of  derived brackets which was invented by   Voronov   \cite{Vor05}. It has been successfully applied to obtain the controlling algebra of relative Rota-Baxter Lie algebra of weight zero \cite{LST,LST2}, the controlling algebra of relative Rota-Baxter algebra of weight zero \cite{DM20} and of Rota-Baxter Lie algebra of arbitrary weight \cite{yj}. However it cannot be applied directly to the general case of extended Rota-Baxter Lie algebras.
     
    The  case for  nonzero weight is more difficult. Wang and Zhou gave the explicit construction of the minimal model of the Rota-Baxter operad and developed the deformation and homotopy theory of Rota-Baxter algebras with weight \cite{WZ24}.
    
    For modified Rota-Baxter algebras, Das \cite{Das22} gave the cohomology theory and applications.  However, two significant challenges, namely the controlling algebra and the homotopy analog of modified Rota-Baxter algebras, are still open. 
    
    We now make a crucial remark on the relationship with existing work. Specifically, the controlling $L_{\infty}[1]$-algebra of Rota-Baxter algebras with weight coincides  with that obtained by Wang and Zhou \cite{WZ24}. This provides further strong evidence, independent of this paper, that the controlling algebra of extended Rota-Baxter algebras is a good one. Comparison with \cite{Das22}: Importantly, the cohomology  we established for modified Rota-Baxter algebras differs from the one proposed by Das. However, the lower cohomology groups in both theories are isomorphic. This commonality in low-degree cohomology is sufficient for applying the theory to study formal deformations and abelian extensions of modified Rota-Baxter algebras, which explains the consistency in those applications despite the underlying theoretical discrepancy. Our results also address two significant challenges: the $L_\infty$-structure and homotopy analogs in the special case of modified Rota-Baxter algebras.

    \subsection{Layout of the paper}\
	
	 We recall some basic notions and  introduce some facts about  extended Rota-Baxter algebras and their modules in Section~\ref{sec 2}. In particular, there is a key descending property for extended Rota-Baxter bimodules in Proposition~\ref{descending property}.
	 
	 In Section \ref{sec 3}, we recall the curved dg Lie algebra and the Maurer-Cartan characterization of associative algebras. We also present the Maurer-Cartan characterization and cohomology of the extended Rota-Baxter operator on associative algebras  in Subsection \ref {subsec 3.2}.
	 
	  In the next section \ref{sec 4}, we recall the $L_\infty[1]$-algebra, since  the controlling algebra of extended Rota-Baxter algebras is an $L_{\infty}[1]$-algebra .
	 
	 The following Section \ref{sec 5} is the key section of our paper. We construct the $L_{\infty}[1]$-algebra of extended Rota-Baxter algebras in Theorem \ref{thm: Linfinity for ass case} and the cohomology of extended Rota-Baxter algebras with arbitrary coefficients in Proposition-Definition \ref {RBA homology on M}.

	 In Section \ref{sec 6}, we study  formal deformations, abelian extensions and homotopy extended Rota-Baxter algebras, as one would expect of a good cohomology theory and of a good  controlling algebra of extended Rota-Baxter algebras.
	 
	 In the last Section \ref{sec 7}, we will see that the controlling algebra is more complex than that of  Rota-Baxter algebras with weight and modified Rota-Baxter algebras. But it seems natural since there exist connections between the controlling $L_{\infty}[1]$-algebra and the eigenvalue of the extended Rota-Baxter operator. We also compare  extended Rota-Baxter algebras and extended Rota-Baxter Lie algebras, and we will see that they have natural connections.
	
	\bigskip
	
	\section{Extended Rota-Baxter algebras}\label{sec 2}
	
    Throughout this paper,	$\bfk$ is a field  with $\mathrm{char}(\bfk)=0$.
    Let $V=\oplus_{n\in \mathbb{Z}} V^n$ be a graded vector space. We define the shift map $s$  by $(sV)^n:=V^{n+1}$. 
    
    Now we introduce extended Rota-Baxter algebras in detail.
    	
	\begin{defn}
		\textbf{An extended Rota-Baxter algebra of weight $(\mu,\lambda)$} for $\mu,\lambda \in \bfk$ is a triple $(A,\pi,T)$ where $(A,\pi)$ is an algebra (we usually omit the notation of  $\pi$)  and  $T:A\to A$ is  a linear map such that $T(a)T(b)=T(T(a)b+aT(b)+\mu ab) +\lambda ab$ for $a,b \in A.$
		
		A morphism of extended Rota-Baxter algebras between $(A,T)$ and $(B,T')$ of the same weight is a morphism of algebras $f: A \to B$ such that $f\circ T= T'\circ f$.
	\end{defn}

   \begin{defn}\label{Def: extended Rota-Baxter bimodules}
   	Let $(A,T)$ be an extended Rota-Baxter algebra of weight $(\mu,\lambda)$ and $M$ be a bimodule
   	over the algebra $A$. We say that $M$ is \textbf{a bimodule
   	over the extended Rota-Baxter algebra} $(A,T)$  or \textbf{an extended Rota-Baxter bimodule} if  $M$ is endowed with a
   	linear map $T_M: M\rightarrow M$ such that the following
   	equations
   	\begin{eqnarray}T(a) T_M(m)&=&T_M\big(aT_M(m)+T(a)m+\mu am\big)+\lambda am,\\
   		T_M(m)T(a)&=&T_M\big(mT(a)+T_M(m)a+\mu ma\big)+\lambda ma.
   	\end{eqnarray}
   	hold for any $a\in A$ and $m\in M$.
   \end{defn}

     \begin{exam}
     	A Rota-Baxter algebra of weight $\lambda$ is precisely an extended Rota-Baxter algebra of weight $(\lambda,0)$.
     	
     	A modified Rota-Baxter algebra of weight $\lambda$ is just an extended Rota-Baxter algebra of weight $(0,\lambda)$.
     \end{exam}
 
     \begin{exam}
	The extended Rota–Baxter algebra $(A,T)$ itself, with $T_M=T$, is an extended Rota–Baxter bimodule over $(A,T)$, it is called the regular extended Rota–Baxter bimodule.
     \end{exam}

   The following is easy to check:
   \begin{prop}  \label{Prop: trivial extension of extended Rota-Baxter bimodule}
   	Let $(A,T)$ be an extended Rota-Baxter algebra of weight $(\mu,\lambda)$ and let $M$ be an extended Rota-Baxter bimodule over $(A,T)$. Then $(A\oplus M, T\oplus T_M)$ is an extended Rota-Baxter algebra with the multiplication defined by \begin{eqnarray}(a,m)(b,n)=(ab, an+mb).\end{eqnarray}
   	
   	Denote it by $A\ltimes  M$, this is called the semi-direct product (or trivial extension) of $A$ by $M$.
   \end{prop}

   Recall first the  following interesting observation:
   \begin{prop}\label{Prop: new RB algebra}
   	Let $(A,T)$ be an extended Rota-Baxter algebra of weight $(\mu,\lambda)$. Define a new binary operation as follows:
   	\begin{eqnarray}a\star  b:=a\cdot T(b)+T(a)\cdot b+\mu a\cdot b\end{eqnarray}
   	for any $a,b\in A$. Then
   the triple  $(A,\star ,T)$ also forms an extended Rota-Baxter algebra of weight $(\mu,\lambda)$  and we denote it by $A_\star.$
   \end{prop}
   	\begin{proof}
   	Notice that $A_\star$ is an  algebra and we leave it to the reader.
   	We only show that $T$ is an extended Rota-Baxter operator on the  algebra $A_\star$.
   	
   	For any $a,b\in A$,
   	\begin{align*}
   	&T(a)\star T(b)\\
   	=& T(a)T^2(b)+T^2(a)T(b)+\mu T(a)T(b)\\
   	=&T\big(aT^2(b)+T(a)T(b)+\mu aT(b)\big)+\lambda aT(b)\\
   	+&T\big(T(a)T(b)+ T^2(a)b+\mu T(a)b\big)+\lambda T(a)b\\
   	+&\mu T\big( T(a)b+ aT(b)+\mu ab\big)+\lambda\mu ab\\
   \end{align*}
   On the other hand,
   \begin{align*}
   	&T\big(T(a)\star b+ a\star T(b)+\mu a\star b\big)+\lambda  a\star b\\
   	=&T\big(T(a)T(b)+ T^2(a)b+\mu  T(a)b\big)\\
   	+&T\big(aT^2(b)+T(a)T(b)+\mu aT(b)\big)\\
   	+&\mu T\big( T(a)b+ aT(b)+\mu ab\big)\\
   	+&\lambda \big( T(a)b+ aT(b)+\mu ab\big)\\
   	=&T(a)\star T(b)
   \end{align*}
   \end{proof}
   
   One can also construct new extended Rota-Baxter bimodules from old ones.
   \begin{prop}\label{descending property}
   	Let $(A,T)$ be an extended Rota-Baxter algebra of weight $(\mu,\lambda)$ and $(M,T_M)$ be an extended Rota-Baxter bimodule over it. Define a left action $``\rhd"$ and a right action $``\lhd"$ of $A$ on $M$ as follows: for any $a\in A,m\in M$,
   	\begin{eqnarray}
   		a\rhd m:&=& T(a)m-T_M(am),\\
   		m\lhd a:&=& mT(a)-T_M(ma).
   	\end{eqnarray}
   	Then these actions make $M$ into an extended Rota-Baxter bimodule over $A_\star $, denote this new  bimodule by $_\rhd M_\lhd$.
   \end{prop}
   
   \begin{proof}
   	Firstly, notice that $_\rhd M_\lhd$ is a bimodule over  the algebra $(A,\star )$ and we leave it as an exercise.
   		
   	Finally, we show that $_\rhd M_\lhd$ is an extended Rota-Baxter bimodule over $A_\star$. That is, for any $a\in A$ and $m\in M$,  $$
   	\begin{array}{rcl} T(a)\rhd T_M(m)&=&T_M\big(a\rhd T_M(m)+T(a)\rhd m+\mu a\rhd m\big)+\lambda a\rhd m,\\
   		T_M(m)\lhd T(a)&=&T_M\big(m \lhd T(a)+T_M(m) \lhd a+\mu m\lhd a\big)+\lambda m\lhd a.\end{array}
   	$$
   	We only prove the first equality, the second being similar.
   	
   	In fact,
   	$$\begin{array}{rcl}
   		T(a)\rhd T_M(m)&=&T^2(a)T_M(m)-T_M(T(a)T_M(m))\\
   		&=&  T_M (T(a)T_M(m) + T^2(a)m +\mu  T(a)m )+\lambda T(a)m-T_M(T(a)T_M(m))  \\
   		&=& T_M\big(   T^2(a)m +\mu  T(a)m \big)+\lambda T(a)m 
   	\end{array}$$
   	and
   	$$\begin{array}{rl}
   		& T_M\big(a\rhd T_M(m)+T(a)\rhd m+\mu a\rhd m\big)+\lambda a\rhd m,\\
   		=&T_M\big( T(a)T_M(m)-T_M(aT_M(m))+ T^2(a)m-T_M(T(a)m)+\mu  T(a)m-\mu T_M(am)\big)+\lambda (T(a) m-T_M(am))\\
   		=& T_M\big( T_M\big(aT_M(m)+T(a)m+\mu am\big)+\lambda am -T_M(aT_M(m))+ T^2(a)m-T_M(T(a)m)\\
   		&+\mu  T(a)m-\mu T_M(am)\big)+\lambda (T(a) m-T_M(am))\\
   		=& T_M\big(   T^2(a)m +\mu T(a)m +\lambda am\big)+\lambda (T(a) m-T_M(am))\\
   		=& T(a)\rhd T_M(m).
   	\end{array}$$
   \end{proof}

	\bigskip
	
	 \section{The MC characterization and the cohomology of associative algebras and of extended Rota-Baxter operators}\ \label{sec 3}
	 
	  As part of extended Rota-Baxter algebras, we now provide a further characterization of associative algebras and extended Rota-Baxter operators.
	  
	  \subsection{Curved dg Lie algebras}\ \label{subsec 3.1}
	 
	 The controlling algebra of extended Rota-Baxter operators  is not a dg Lie algebra but a curved dg Lie algebra. So it is necessary to introduce it.
	 
	 \begin{defn}\label{curved dg Lie}
	 	\textbf{A curved dg Lie algebra} is a quadruple $(\frakg,d_0,d_1,d_2)$ where $\frakg$ is a graded space and three graded linear operators $d_0:\bfk \rightarrow \frakg,d_1:\frakg\rightarrow \frakg$ and $d_2:\frakg\otimes \frakg\rightarrow \frakg$  satisfy the following conditions:\\
	 	$(0) \quad d_2(x,y)=(-1)^{|x||y|} d_2(y,x)$,\\
	 	$(1) \quad d_1\circ d_0=0$,\\
	 	$(2) \quad d_1\circ d_1(x) +d_2(d_0(1), x)=0$,\\
	 	$(3)\quad d_1\circ d_2+d_2\circ (d_1\otimes Id+ Id\otimes d_1)=0$,\\
	 	$(4) \quad d_2(d_2(x,y),z)+(-1)^{|x||y|+|x||z|}d_2(d_2(y,z),x)+(-1)^{|y||z|}d_2(d_2(x,z),y)=0$ (Jacobi identity),\\
	 	for homogeneous elements $x,y,z\in \frakg$.
	 \end{defn}
	 
	 \begin{rmk}
	 			If $d_0=0$, it's just a dg Lie algebra. Moreover, if $d_0=0,d_1=0$, it's just a graded Lie algebra. Note that it  is slightly different from the definition of ordinary graded Lie algebras. 
	 		
	 		\textbf{An ordinary graded Lie algebra} is a pair $(\frakg,[-,-])$ where $\frakg$ is a graded space with the bracket $[-,-]:\frakg\otimes \frakg\rightarrow \frakg$ that satisfies the following conditions :\\
	 		$(1) \quad [x,y]=-(-1)^{|x||y|} [y,x]$,\\
	 		$(2) \quad [[x,y],z]+(-1)^{|x||y|+|x||z|}[[y,z],x]-(-1)^{|y||z|}[[x,z],y]=0$ (Jacobi identity),\\
	 		for homogeneous elements $x,y,z\in \frakg$.
	 		
	 		Given an ordinary graded  Lie algebra  $(\frakg,[-,-])$, then $(s\frakg,\pi)$ is a graded  Lie algebra where $\pi(sx,sy)=(-1)^{|x|}s[x,y]$ for  homogeneous elements $x,y\in \frakg$.
	 \end{rmk}
	 
	 \begin{defn}
	 	Given a curved dg Lie algebra $(\frakg,d_0,d_1,d_2)$. An element $\alpha\in \frakg^0$ is called a Maurer-Cartan element if and only if it satisfies the Maurer-Cartan equation:\\
	 	$$d_0(1)+d_1(\alpha)+\frac{1}{2}d_2(\alpha,\alpha)=0.$$
	 \end{defn}
	 
	 \begin{prop}[Twisting procedure] \label{Prop: twist curved dg Lie}
	 	Let  $\alpha$ be a Maurer-Cartan element of the curved dg Lie algebra $\frakg$. The twisted dg Lie algebra on $\frakg$  is given by $d_n ^{\alpha}: \frakg^{\ot n}\rightarrow \frakg$ for $n=1,2$ which is defined as follows$\colon$
	 	\begin{eqnarray*}\label{Eq: twisted curved dg Lie algebra}  \quad d^\alpha_1(x)=d_{1} (x)+d_2(\alpha,x),\quad d^\alpha_2(x,y)=d_2(x,y),\ \forall x,y\in \frakg.\end{eqnarray*}
	 	Denote $\mathcal{MC}(\frakg):=\{\mbox{Maurer-Cartan elements of}~ \frakg\}$.
	 \end{prop}
 
     \medskip
    
    \subsection{The MC characterization and the cohomology of associative algebras}\ \label{subsec 3.2}
    \begin{defn}
    	Let $A$ be a vector space. Consider the graded space $\mathrm{Hom}(\T(A),A):=\oplus_{i\ge 0}\mathrm{Hom}(A^{\otimes i},A) $ and we say the degree  of $f$ is $m$ if $f\in \mathrm{Hom}(A^{\otimes m+1},A)$. The \textbf{partial Gerstenhaber composition} of two operators  $f$ and $g$ on position $i$ is defined to be$$f\bar{\circ}_i g(x_1,\dots,x_{m+n+1}):= f(x_1,\dots,x_{i-1},g(x_{i},\dots,x_{i+n}),x_{i+n+1},\dots,x_{m+n+1})$$ for $f\in \mathrm{Hom}(A^{\otimes m+1},A)$, $g\in \mathrm{Hom}(A^{\otimes n+1},A)$, $x_1,\dots,x_{m+n+1}\in A$ and $1\le i \le m+1$.
    	
    	 The \textbf{Gerstenhaber composition} of $f$ and $g$ is given by $f\bar{\circ} g=\sum_{i=1}^{m+1} (-1)^{(i-1)n}f\bar{\circ}_i g$ and the \textbf{Gerstenhaber bracket} $[-,-]_{\bf{G}}$ defined on $\mathrm{Hom}(\T(A),A)$ is given by $[f,g]_{\bf{G}}=f\bar{\circ} g-(-1)^{mn}g\bar{\circ} f$. See \cite{Ge1,Ge2} for more details.
    \end{defn}
    
    \begin{prop}\
    	
    	\begin{enumerate}
    		\item Let $A$ be a vector space. Then $(s\mathrm{Hom}(\T(A),A),[\cdot,\cdot])$ forms a graded Lie algebra where we define $[sx,sy]:=(-1)^{|x|}s[x,y]_{\bf{G}}$ for any $x\in \mathrm{Hom}(A^{\otimes p+1},A)$ and $y\in \mathrm{Hom}(A^{\otimes q+1},A)$.
    		\item An element $s\pi$ is a Maurer-Cartan element  of $s\mathrm{Hom}(\T(A),A)$ if and only if $\pi$ is an associative  product on $A$.
    	\end{enumerate}
    \end{prop}
    
    \begin{rmk}
    	$(\mathrm{Hom}(\T(A),A),[\cdot,\cdot]_{\bf{G}})$ is an ordinary graded Lie algebra by \cite{Ge1}. It follows that the above proposition holds. 
    \end{rmk}

    Let $(A,\pi)$ be an associative algebra. Then $s\pi$  is a Maurer-Cartan element of $(s\mathrm{Hom}(\T(A),A),[\cdot,\cdot])$, hence twisting by $s\pi$, we have a new dg Lie algebra. 
    \begin{defn}\label{cohomology of algebra}
    	Let $(A,\pi)$ be an associative algebra. \textbf{The cochain complex of $A$} is defined to be $(\mathrm{Hom}(\T(A),A),\partial_{\Alg}:=[\pi,-]_{\bf{G}})$ induced by the new dg Lie algebra above. The corresponding cohomology is called \textbf{the cohomology of $A$}.  We could see that they are just the Hochschild cochain complex and Hochschild cohomology of $A$. More precisely,\\
    	
    	for   $x_1,\dots,x_{n}\in A$ and $f\in \Hom(A^{\otimes n-1},A)$, we have
    	\begin{eqnarray*}
    		&&\partial_{\Alg}(f)(x_1, \dots, x_{n})\\
    		&=&(-1)^n x_1 f(x_2,\dots, x_n)+f(x_1,\dots x_{n-1}) x_n\\
    		&&+\sum_{i=1}^{n-1} (-1)^{i+n}f( x_1,\dots,x_{i-1},x_i x_{i+1}, \dots,x_n).
    	\end{eqnarray*}
    \end{defn}

    \medskip
    
    \subsection{The MC characterization and the cohomology of extended Rota-Baxter operators}\
    
    Let $(A,\pi)$ be an algebra.
    Define three graded linear operators as follows, for $f\in\mathrm{Hom}(A^{\otimes n+1},A), g\in \mathrm{Hom}(A^{\otimes m+1},A)$,
    
    $d_0:\bfk \rightarrow \mathrm{Hom}(\T(A),A)$ is given by $d_0(1)=-\lambda \pi$,\\
    $d_1=d:\mathrm{Hom}(\T(A),A)\rightarrow \mathrm{Hom}(\T(A),A)$ is given by $d_1(f)=(-1)^{|f|+1}\mu f\bar{\circ} \pi$,\\
    and $d_2=[-,-]:\mathrm{Hom}(\T(A),A)\otimes \mathrm{Hom}(\T(A),A)\rightarrow \mathrm{Hom}(\T(A),A)$ is given by \\
    $d_2(f,g)(x_1,\dots,x_{m+n+2})=(-1)^{(n+1)m}\pi(f(x_{1},\dots,x_{n+1}),g(x_{n+2},\dots,x_{n+m+2}))$\\
    $+(-1)^{n}\pi(g(x_{1},\dots,x_{m+1}),f(x_{m+2},\dots,x_{n+m+2}))$\\
    $-\sum_{i=1}^{m}(-1)^{(n+1)(m+i-1)} g(x_1,\dots,x_{i-1},\pi(f(x_{i},\dots,x_{i+n}),x_{i+n+1}),\dots,x_{n+m+2})$\\
    $-\sum_{i=1}^{m}(-1)^{(n+1)(m+i-1)+n} g(x_1,\dots,x_{i-1},\pi(x_i,f(x_{i+1},\dots,x_{i+n+1})),\dots,x_{n+m+2})$\\
    $-\sum_{i=1}^{n}(-1)^{n+(m+1)(i-1)} f(x_1,\dots,x_{i-1},\pi(g(x_{i},\dots,x_{i+m}),x_{i+m+1}),x_{i+m+2},\dots,x_{n+m+2})$\\
    $-\sum_{i=1}^{n}(-1)^{n+(m+1)(i-1)+m} f(x_1,\dots,x_{i-1},\pi(x_i,g(x_{i+1},\dots,x_{i+m+1})),x_{i+m+2},\dots,x_{n+m+2})$.

    \begin{thm}
    	With the above notation, then $(\mathrm{Hom}(\T(A),A),d_0,d, [-,-])$ forms a curved dg Lie algebra.
    \end{thm}
    \begin{rmk}
    	The above theorem can be proved by direct computation. Here we will give an alternative proof in Corollary \ref{curved dg Lie of operator} by using
    	Theorem \ref{thm: Linfinity for ass case}.
    \end{rmk}

    \begin{thm}\label{MC elements ass}
    	With the above notation.  Let $T\in\mathrm{Hom}(A,A)$, then
    	$T\in \mathcal{MC}(\mathrm{Hom}(\T(A),A))$ if and only if $T$ is  an extended Rota-Baxter operator of weight  $(\mu,\lambda)$ on $A$.
    \end{thm}
    \begin{proof}
    	\begin{equation*}
    		\begin{aligned}
    			0 &= 	d_0(1)+d(T)+\frac{1}{2}[T,T]  \\
    			&=  -\lambda \pi-\mu T\circ \pi+ \frac{1}{2} [T,T].
    		\end{aligned}
    	\end{equation*}
    	This equation of $T$ coincides with the one obtained in the proof of Theorem \ref{Thm: MC elements in ex Linifnity}
    \end{proof}
    
    By Theorem~\ref{MC elements ass}, an extended Rota-Baxter operator $T$ is a Maurer-Cartan element in the curved dg Lie algebra $\mathrm{Hom}(\T(A),A)$. Twisting $\mathrm{Hom}(\T(A),A)$ by  $T$ yields a new differential.   
    
    \begin{prop-def}\label{cohomology of operator}
    	Consider an algebra $A$. \textbf{The cochain complex of the extended Rota-Baxter operator  $T$} is defined to be $(\mathrm{Hom}(\T(A),A),\partial_\RB:=d_{1}^{T})$. The corresponding  cohomology group is called \textbf{the cohomology of the extended Rota-Baxter operator  $T$}. One may observe that this complex (and cohomology group) coincides with the Hochschild cochain complex (and cohomology group)  of $A_\star$ with coefficients in $_\rhd A_\lhd$(See Proposition  \ref{Prop: new RB algebra} and \ref{descending property}).
    \end{prop-def}

    \begin{proof}
    	For   $x_1,\dots,x_{n}\in A$ and $f\in \Hom(A^{\otimes n-1},A)$, we have
    	\begin{eqnarray*}
    		&&(d_1(f)+d_2(T, f))(x_1, \dots, x_{n})\\
    		&=&\sum_{i=1}^{n-1} (-1)^{i+n}\mu f( x_1,\dots,x_ix_{i+1} \dots,x_n)+(-1)^n T(x_1) f(x_2,\dots, x_n)\\
    		&& +f(x_1,\dots, x_{n-1})T(x_n)-(-1)^{n}T(x_1 f(x_2,\dots x_n))-T(f(x_1,\dots x_{n-1})x_n)\\
    		&&+\sum_{i=1}^{n-1} (-1)^{i+n}f(x_1,\dots,(T(x_i)x_{i+1}+x_iT(x_{i+1})), \dots,x_n)\\
    		&=&(-1)^n x_1 \rhd f(x_2,\dots, x_n)+f(x_1,\dots x_{n-1})\lhd x_n\\
    		&&+\sum_{i=1}^{n-1} (-1)^{i+n}f( x_1,\dots,x_{i-1},x_i\star x_{i+1}, \dots,x_n)\\
    	\end{eqnarray*}
    	induces a new differential on $\mathrm{Hom}(\T(A),A)$ which differs from the usual Hochschild cochain complex. And  they are just the Hochschild cochain complex and cohomology of $A_\star$ with coefficient in $_\rhd A_\lhd$. 
    \end{proof}

   \bigskip

 \section{$L_\infty[1]$-algebras}\ \label{Subsect: Linfinity algebras}\label{sec 4}
 
 In this section, we  recall some preliminaries on $L_\infty[1]$-algebras. It is well known that $L_\infty[1]$-algebras are equivalent to $L_\infty$-algebras. Consequently, all results parallel to those for $L_\infty$-algebras hold. For $L_\infty$- algebras, see \cite{Get09} for more details.
 
 Let $V=\oplus_{n\in \mathbb{Z}} V^n$ be a graded vector space. Recall that   the graded symmetric algebra $\rmS(V)$ of $V$ is defined to be the quotient of the tensor algebra $\T(V)$ by   the two-sided ideal $I$   generated by
 $x\ot y -(-1)^{|x||y|}y\ot x$ for all homogeneous elements $x, y\in V$. For $x_1\ot\cdots\ot x_n\in V^{\ot n}\subseteq \T(V)$, write $ x_1\odot x_2\odot\dots\odot x_n$ for its image in $\rmS(V)$.
 For homogeneous elements $x_1,\dots,x_n \in V$ and $\sigma\in S_n$ which is  the symmetric group in $n$ variables, the Koszul sign $\varepsilon(\sigma):=\varepsilon(\sigma;  x_1,\dots, x_n)$ is defined by
 $$ x_1\odot x_2\odot\dots\odot x_n=\varepsilon(\sigma)x_{\sigma(1)}\odot x_{\sigma(2)}\odot\dots\odot x_{\sigma(n)}\in \rmS(V).$$
 Denote by $\rmS^n(V)$ the image of $V^{\otimes n}$ in $\rmS(V)$.
 
 Let  $n\geq 1$.
 For $0\leq i_1, \dots, i_r\leq n$ with $i_1+\cdots+i_r=n$,  $\Sh(i_1, i_2,\dots,i_r)$ is the   set of $(i_1,\dots, i_r)$-shuffles, i.e., the permutation $\sigma\in S_n$ such that
 $$\sigma(1)<\sigma(2)<\dots<\sigma(i_1),  \ \sigma(i_1+1)< \dots<\sigma(i_1+i_2),\ \dots,\
 \sigma(i_1+\cdots+i_{r-1}+1)< \cdots<\sigma(n).$$
 
 \smallskip
 \begin{defn}\label{Def:L[1]-infty}
 	An \textbf{$L_\infty[1]$-algebra} is a graded vector space  $\frakg=\bigoplus\limits_{i\in\mathbb{Z}}\frakg^i$   equipped with   a family of graded linear maps $l_n:\frakg^{\ot n}\rightarrow \frakg, n\geq 1$ of degree $1$ satisfying  the following equations:
 	for arbitrary  $n\geq 1$,  $ \sigma\in \rmS_n$ and $x_1,\dots, x_n\in \frakg$,
 	\begin{enumerate}
 		\item[(i)](graded symmetry)
 		\begin{equation*} \label{graded sym}
 			l_n(x_{\sigma(1)},\dots,x_{\sigma(n)})=\varepsilon(\sigma)l_n(x_1,\dots,x_n),
 		\end{equation*}

 		\item[(ii)](generalised Jacobi identity)
 		\begin{equation*}\label{graded Jacobi}
 			\sum_{i=1}^n\sum_{\sigma\in \Sh(i,n-i)}\varepsilon(\sigma)l_{n-i+1}(l_i(x_{\sigma(1)},\dots,x_{\sigma(i)}),x_{\sigma(i+1)},\dots,x_{\sigma(n)})=0.
 		\end{equation*}
 		
 	\end{enumerate}
 	
 \end{defn}

	\begin{rmk} \label{Rem: L[1]-infinity for small n}   Let us consider the generalised Jacobi identity for   $n\leq 3$ with the assumption of  generalised  symmetry.
		
		\begin{enumerate}
			\item[(i)]  For $n=1$,  then $l_1 \circ l_1 =0$, that is,  $l_1 $ is a differential.

			\item[(ii)] For $n=2$, then $l_1\circ l_2 +l_2 \circ (l_1 \ot\Id+\Id\ot l_1)=0$, that is , $l_1$ is a derivation with respect to $l_2$.

			\item[(iii)] For $n=3$ and arbitrary homogeneous elements $x_1, x_2, x_3\in \frakg$, we have
			$$\begin{array}{ll} &l_2 (l_2 (x_1\ot x_2)\ot x_3)+(-1)^{|x_1|(|x_2|+|x_3|)} l_2 (l_2 (x_2\ot x_3)\ot x_1)+
				(-1)^{|x_3|(|x_1|+|x_2|)} l_2 (l_2 (x_3\ot x_1)\ot x_2)
				\\
				=&-\Big(l_1 (l_3 (x_1\ot x_2\ot x_3))+ l_3 (l_1  (x_1)\ot x_2\ot x_3 )+(-1)^{|x_1|} l_3 (x_1\ot l_1  (x_2)\ot x_3 )+\\
				&(-1)^{|x_1|+|x_2|} l_3 (x_1\ot x_2\ot l_1  (x_3) )\Big),\end{array}$$
			that is, $l_2$ satisfies the   Jacobi identity up to homotopy.
		\end{enumerate}

	\end{rmk}

	\begin{defn}
		A \textbf{Maurer-Cartan element} of an $L_\infty[1]$-algebra $(\frakg,\{l_n \}_{n\geq1})$ is  an element $\alpha\in \frakg^{0}$   satisfying the Maurer-Cartan equation:
		\begin{eqnarray*}\label{Eq: mc-equation[1]}\sum_{n=1}^\infty\frac{1}{n!} l_n (\alpha^{\ot n})=0,\end{eqnarray*}
		whenever this infinite sum exists.  Denote $\mathcal{MC}(\frakg):=\{\mbox{Maurer-Cartan elements of}~ \frakg\}$.
	\end{defn}

	\begin{prop}[Twisting procedure] \label{Prop: twist-L-infty[1]}
		Let  $\alpha$ be a Maurer-Cartan element of $L_\infty[1]$-algebra $\frakg$. The twisted $L_\infty[1]$-algebra  is given by $l_n ^{\alpha}: \frakg^{\ot n}\rightarrow \frakg$ which is defined as follows$\colon$
		\begin{eqnarray*}\label{Eq: twisted L[1] infinity algebra} l^\alpha_n(x_1\ot \cdots\ot x_n)=\sum_{i=0}^\infty\frac{1}{i!}l_{n+i} (\alpha^{\ot i}\ot x_1\ot \cdots\ot x_n),\ \forall x_1, \dots, x_n\in \frakg,\end{eqnarray*}
		whenever these infinite sums exist.
	\end{prop}
	
    The following discussion concerns the relationship between two $L_\infty[1]$-algebras under a strict  $L_\infty[1]$-morphism.
    
   \begin{defn}
   	\textbf{A strict $L_\infty[1]$-morphism} $f$ between $L_\infty[1]$-algebras $(\frakg,\{l_i\})$ and $(\frakh,\{L_i\})$ is a  linear map $f: \frakg \to \frakh$ of degree $0$ such that
   	$$f_1(l_n(x_1,\dots,x_n))=L_n(f_1(x_1),\dots,f_1(x_n))$$
   		for all $n \ge 1$ and homogeneous elements $x_1,\dots, x_n\in \frakg$.
   	\end{defn}

	We then have the following observation:
	\begin{prop}\label{relations}
		With the above notation. Let  $\alpha$ be a Maurer-Cartan element of $L_\infty[1]$-algebra $\frakg$. Then $f_1(\alpha)$ is a Maurer-Cartan element of $\frakh$. And there exists an induced strict $L_\infty[1]$-morphism  between the twisted $L_\infty[1]$-algebras $(\frakg,\{l_n ^{\alpha}\})$ and $(\frakh,\{L_n ^{f_1(\alpha)}\})$.
	\end{prop}

	\bigskip
	\section{The $L_{\infty}$-structure and cohomology theory of extended Rota-Baxter algebras with weight}\label{sec 5}
	
	Now we will present the controlling algebra and the cohomology of extended Rota-Baxter algebras.
	
	\subsection{The $L_{\infty}$-structure for extended Rota-Baxter algebras with weight}\ \label{subsec 5.1}
	 
	Before constructing the $L_{\infty}$- algebra, we give some facts about the field $\bfk$.  The following is an easy exercise:
	\begin{lem} \label{key lemma}
		For $\mu,\lambda \in \bfk$ and $m,n \in \mathbb{Z}$, set $t=\sqrt{\mu^2-4\lambda}$. We have the identity whenever the following  elements exist:
		\begin{equation*}
			\begin{aligned}
				&\frac{\lambda}{t}((\frac{\mu +t}{2})^{m}-(\frac{\mu -t}{2})^{m})\frac{\lambda}{t}((\frac{\mu +t}{2})^{n}-(\frac{\mu -t}{2})^{n})+\frac{\lambda}{t}((\frac{\mu +t}{2})^{n+m+1}-(\frac{\mu -t}{2})^{n+m+1})\\
				&=\frac{\lambda}{t}((\frac{\mu +t}{2})^{m+1}-(\frac{\mu -t}{2})^{m+1})\frac{1}{t}((\frac{\mu +t}{2})^{n+1}-(\frac{\mu -t}{2})^{n+1})	
			\end{aligned}
		\end{equation*}
	\end{lem}

	\begin{rmk}\label{coefficient}
		We will denote $A_{n}:=\frac{\lambda}{t}((\frac{\mu +t}{2})^{n}-(\frac{\mu -t}{2})^{n}) $ and $B_{n}:=\frac{1}{t}((\frac{\mu +t}{2})^{n}-(\frac{\mu -t}{2})^{n})$. 
		
		We only consider $A_n $ for $n\ge -1$ and $B_n $ for $n\ge 0$ in the following context. Note that we are Not concerned with the coefficients $A_n$($B_n$) when $t=0$ or the existence of $t$. We use it as a formal notation, what we actually use is  the explicit expression for each term just as given below. Then $A_mA_n+A_{m+n+1}=A_{m+1}B_{n+1}$( or $A_mB_n+B_{m+n+1}=B_{m+1}B_{n+1}$) without any assumption.
		
		 For small $n$, we have\\
		$A_{0}=0$, $A_{1}=\lambda$,\\
		$A_{2}=\lambda \mu$, $A_{3}=\lambda (\mu^2-\lambda)$,\\
		$A_{4}=\lambda (\mu^3-2\mu\lambda)$. In particular, we assume that	$A_{-1}=-1$ even if $\lambda=0$.

	\end{rmk}
    
    Let $A$ be a vector space.  Set $$\frakm:=\mathrm{Hom}(\T(A),A)\ \mathrm{and}\
    \fraka:=\Hom(\T(A),A).$$
    
	\begin{thm}\label{thm: Linfinity for ass case}
		Keep the above notation and denote $t=\sqrt{\mu^2-4\lambda}$. There exists an $L_\infty[1]$-algebra structure on  $s \mathfrak{m}\oplus\fraka$, where $l_i$ are given by 
		$$l_1(sf)=-A_n f$$
		$$
		l_2(sf,sg)      =    (-1)^{|f|} s[f,g]_{\G}, $$
		when $i\geq 2,n=i-2$, we have
		\begin{equation*}
			\begin{aligned}
				l_i(sf,\xi_1,\cdots,\xi_{i-1})=&\sum\limits_{\substack{\sigma\in \rmS_{i-1}}}\epsilon(\sigma)\\ &
				f(\xi_{\sigma(1)}\otimes\dots\otimes\xi_{\sigma(i-1)})\\
				&-\sum\limits_{\substack{\sigma\in \rmS_{i-1}}}\sum_{k=1}^{n+1} \epsilon(\sigma) (-1)^{nm_{\sigma(1)}} \\ 
				& \xi_{\sigma(1)}\bar{\circ}  f(\xi_{\sigma(2)}\otimes\dots\otimes\underline{\id}_{kth}  \ot\dots\otimes\xi_{\sigma(i-1)}).
			\end{aligned}
		\end{equation*}
		and for $i\geq 2$, $n> i-2$,
		\begin{equation*}
			\begin{aligned}
				l_i(sf,\xi_1,\cdots,\xi_{i-1})  =&-\sum\limits_{\substack{\sigma\in \rmS_{i-1}\\a_j+m_{\sigma(j)}+1\le a_{j+1}}}\epsilon(\sigma)\\ &
				A_{n-i+1} \big(\cdots ((f\bar{\circ}\xi_{\sigma(1)})\bar{\circ}_{a_2}\xi_{\sigma(2)})\bar{\circ}_{a_3}\cdots\big)\bar{\circ}_{a_{i-1}}\xi_{\sigma(i-1)}\\
				&-\sum\limits_{\substack{\sigma\in \rmS_{i-1}\\a_j+m_{\sigma(j)}+1\le a_{j+1}, j\ge 2}} \epsilon(\sigma) (-1)^{nm_{\sigma(1)}} \\ &
				B_{n-i+3} \xi_{\sigma(1)}\bar{\circ}_{a_1}\big(\cdots ((f\bar{\circ}_{a_2}\xi_{\sigma(2)})\bar{\circ}_{a_3}\cdots)\bar{\circ}_{a_{i-1}}\xi_{\sigma(i-1)}\big),
			\end{aligned}
		\end{equation*}
		
		for homogeneous elements	$f\in\Hom(A^{\otimes n+1},A)\subseteq \frakm$, $g\in\Hom(A^{\otimes m+1},A)\subseteq \frakm$ and $\xi_j\in\Hom(A^{\otimes m_j+1},A)\subseteq \fraka$, $1\leq j\leq i-1$, and all other components vanish.
	\end{thm}

	\begin{rmk}
		We can write $l_i(sf,\xi_1,\cdots,\xi_{i-1})$ for $i\ge 1$ in the same form. The proof will be given in the Appendix. Although somewhat involved, we can give a brief summary here:  the generalised Jacobi identity holds  by Lemma \ref{key lemma} and the fact that some terms  appear twice with opposite signs.
	\end{rmk}

	\begin{thm}\label{Thm: MC elements in ex Linifnity}
		With the notation above. Let $\pi\in\mathrm{Hom}(A^{\otimes 2},A)$, $T\in\mathrm{Hom}(A,A)$, then
		$(s \pi,T)\in \mathcal{MC}(s \frakm\oplus\fraka)$ if and only if $(A,\pi,T)$ is  an extended Rota-Baxter algebra  of weight  $(\mu,\lambda)$.
	\end{thm}

	\begin{proof}
		
		$(s\pi,T) \in\mathcal{MC}(s \frakm\oplus\fraka)$ if and only if
		$[\pi,\pi]_{\G}=0$ (that is, $\pi$ is an associative product) and
		\begin{equation*}
			\begin{aligned}
				0 &= 	l_1(s\pi)+\sum_{k=2}^\infty\frac{1}{(k-1)!}l_k(s\pi,\underbrace{T,\dots,T}_{(k-1)\ \mathrm{times}})  \\
				&=  	l_1(s\pi)+\mu  l_2(s\pi,T)+\frac{1}{2}l_3(s\pi,T,T)\\
				&=  -\lambda\pi-\mu T\bar{\circ} \pi+\frac{1}{2}l_3(s\pi,T,T),
			\end{aligned}
		\end{equation*}
		For arbitrary  $(x,y)\in A^{\otimes 2}$, we have
		\begin{equation*}
			\begin{aligned}
				0&=	\big(-\lambda\pi- \mu T\bar{\circ} \pi+\frac{1}{2}l_3(s\pi,T,T)\big)(x,y)\\
				&=-\lambda\pi(x,y)-\mu T\circ \pi(x,y)+\pi(Tx,Ty)-T\pi(x,Ty)-T\pi(Tx,y)\\
				&=-\lambda xy-\mu T(xy)+T(x)T(y)-T(T(x)y)- T(xT(y))\\
				&=T(x)T(y)-T(T(x)y- xT(y)- \mu xy)-\lambda xy.\\
			\end{aligned}
		\end{equation*}
		Hence $T$ is an extended Rota-Baxter operator of weight  $(\mu,\lambda)$.
	\end{proof}

    By Theorem \ref{thm: Linfinity for ass case}, we have
    
	\begin{cor}\label{curved dg Lie of operator}
		 With the above notation. Let $(A,\pi)$ be an associative algebra, substitute $s\pi$ in place of $s f$ in the maps $l_i(sf,\xi_1,\cdots,\xi_{i-1})$ for $i\geq 1$.  Then $\fraka:=\Hom(\T(A),A)$ equipped with $d_0(1)=l_1(s\pi), d_1(-)=l_2(s\pi,-),d_2(-,-)=l_3(s\pi,-,-)$ is a curved dg Lie algebra.  
	\end{cor}

	\medskip
	
	\subsection{Applications for special cases: Rota-Baxter algebras with weight and modified Rota-Baxter algebras}\label{subsec 5.2}\
	
	We now discuss two significant algebraic structures: Rota-Baxter algebras with weight and modified Rota-Baxter algebras.
	
	\begin{cor}\label{wRBLie}
		Keep the above notation and let $\lambda=0$, then there exists an $L_\infty[1]$-algebra structure on  $s \frakm\oplus\fraka$. It is just the  $L_\infty[1]$-algebra of  Rota-Baxter algebras of weight $\mu$, where $l_i$ are given by :

		$$
		l_2(sf,sg)      =    (-1)^{|f|} s[f,g]_{\G}, $$
		and for $i\geq 2,n=i-2$,
		\begin{equation*}
			\begin{aligned}
				l_i(sf,\xi_1,\cdots,\xi_{i-1})=&\sum\limits_{\substack{\sigma\in \rmS_{i-1}}}\epsilon(\sigma)\\ &
				f(\xi_{\sigma(1)}\otimes\dots\otimes\xi_{\sigma(i-1)})\\
				&+\sum\limits_{\substack{\sigma\in \rmS_{i-1}}}\sum_{k=1}^{n+1} \epsilon(\sigma) (-1)^{nm_{\sigma(1)}} \\ 
				&- \xi_{\sigma(1)}\bar{\circ}  f(\xi_{\sigma(2)}\otimes\dots\otimes\underline{\id}_{kth}  \ot\dots\otimes\xi_{\sigma(i-1)}),
			\end{aligned}
		\end{equation*}
		for $i\geq 2,n>i-2$,
		\begin{equation*}
			\begin{aligned}
				&l_i(sf,\xi_1,\cdots,\xi_{i-1})=
				\sum\limits_{\substack{\sigma\in \rmS_{i-1}\\a_j+m_{\sigma(j)}+1\le a_{j+1}, j\ge 2}} \epsilon(\sigma) (-1)^{nm_{\sigma(1)}} \\ &
				-\mu^{n-i+2} \xi_{\sigma(1)}\bar{\circ}\big(\cdots ((f\bar{\circ}_{a_2}\xi_{\sigma(2)})\bar{\circ}_{a_3}\cdots)\bar{\circ}_{a_{i-1}}\xi_{\sigma(i-1)}\big).
			\end{aligned}
		\end{equation*}
		
		For homogeneous elements	$f\in\Hom(A^{\otimes n+1},A)\subseteq \frakm$, $g\in\Hom(A^{\otimes m+1},A)\subseteq \frakm$	and $\xi_j\in\Hom(A^{\otimes m_j+1},A)\subseteq \fraka$, $1\leq j\leq i-1$, and all other components vanish.
	\end{cor}

	\begin{cor}\label{mRNLie}
		Keep the above notation and let $\mu=0$, so  $t=2\sqrt{-\lambda}$, then there exists an $L_\infty[1]$-algebra structure on  $s \frakm\oplus\fraka$. It is precisely the  $L_\infty[1]$-algebra of modified Rota-Baxter algebras, where $l_i$ are given by :
		
		\begin{eqnarray}
			l_1(sf)=
			\begin{cases}
				(-\lambda)^{\frac{n+1}{2}} f	, &  \text{$n$ is odd }  \\
				0, & \text{$n$ is even},
			\end{cases}
		\end{eqnarray}
		$$
		l_2(sf,sg)      =    (-1)^{|f|} s[f,g]_{\G}, $$
		and for $i\geq 2$,
		when $n-i$ is even, 
		\begin{equation*}
			\begin{aligned}
				l_i(sf,\xi_1,\cdots,\xi_{i-1})=&\sum\limits_{\substack{\sigma\in \rmS_{i-1}\\a_j+m_{\sigma(j)}+1\le a_{j+1}}}\epsilon(\sigma)\\ &
				(-\lambda)^{\frac{n-i+2}{2}}\big(\cdots ((f\bar{\circ}_{a_1}\xi_{\sigma(1)})\bar{\circ}_{a_2}\xi_{\sigma(2)})\bar{\circ}_{a_3}\cdots\big)\bar{\circ}_{a_{i-1}}\xi_{\sigma(i-1)}\\
				&+\sum\limits_{\substack{\sigma\in \rmS_{i-1}\\a_j+m_{\sigma(j)}+1\le a_{j+1}, j\ge 2}} \epsilon(\sigma) (-1)^{nm_{\sigma(1)}} \\ 
				&(-\lambda)^{\frac{n-i+2}{2}} \xi_{\sigma(1)}\bar{\circ}\big(\cdots ((f\bar{\circ}_{a_2}\xi_{\sigma(2)})\bar{\circ}_{a_3}\cdots)\bar{\circ}_{a_{i-1}}\xi_{\sigma(i-1)}\big),    \\ 
			\end{aligned}
		\end{equation*}\\
		when $n-i$ is odd, $l_i(sf,\xi_1,\cdots,\xi_{i-1})=0$,
		for homogeneous elements	$f\in\Hom(A^{\otimes n+1},A)\subseteq \frakm$, $g\in\Hom(A^{\otimes m+1},A)\subseteq \frakm$	and $\xi_j\in\Hom(A^{\otimes m_j+1},A)\subseteq \fraka$, $1\leq j\leq i-1$, and all other components vanish.
	\end{cor}

	\medskip
	
	\subsection{The cohomology of extended Rota-Baxter algebras }\ \label{subsec 5.3}

	Let $A$ be a vector space.  Let $$\frakm:=\mathrm{Hom}(\T(A),A)=\bigoplus_{n\ge 0}\Hom(A^{\otimes n},A)\ \mathrm{and}\
	\fraka:=\Hom(\T(A),A)=\bigoplus_{n\ge 0}\Hom(A^{\otimes n},A).$$
	Recall that we have constructed an $L_\infty[1]$-structure on $s\frakm\oplus\fraka$ in Theorem~\ref{thm: Linfinity for ass case}.

	Let $(A, \pi, T)$ be an extended Rota-Baxter  algebra.  By Theorem~\ref{Thm: MC elements in ex Linifnity}, $(s\pi, T)$ is a Maurer-Cartan element in the $L_\infty[1]$-algebra $s\frakm\oplus\fraka$. By  Proposition~\ref{Prop: twist-L-infty[1]}, twisting $s\frakm\oplus\fraka$ by  $(s\pi, T)$ gives a new $L_\infty[1]$-algebra, whose new differential is denoted by $l_{1}^{(s\pi,T)}$.

	\begin{prop}\label{cochaincomplexad}
		The complex $(s\frakm\oplus\fraka,l_{1}^{(s\pi,T)})$ is given by a mapping cone.
	\end{prop}
	\begin{proof}

		It suffices to make explicit the differential $l_1^{(s\pi, T)}$.
		
		For  $n\geq 1$, $f\in \mathrm{Hom}( A^{\otimes n},A), g\in  \mathrm{Hom}( A^{\otimes n-1},A)$,
		\begin{eqnarray*}
			l_1^{(s\pi, T)}(sf, g)
			&=& \sum_{k=0}^{\infty}\frac{1}{k!}l_{k+1}(\underbrace{(s\pi,T),\cdots,(s\pi,T)}_{k\ \mathrm{times}}, (sf, g))\\
			&=&l_1(sf,g)+ l_2((s\pi, T),(sf,g))
			+ \sum_{k=2}^{\infty}\frac{1}{k!}l_{k+1}(\underbrace{(s\pi,T),\cdots,(s\pi,T)}_{k\ \mathrm{times}}, (sf, g))\\
			&=& \big(l_2(s\pi, sf), l_1(sf)+l_2(s\pi, g)+l_3(s\pi, T, g)+l_2(sf, T)
			+ \sum_{k=2}^{n}\frac{1}{k!}l_{k+1}(sf, \underbrace{T,\cdots,T}_{k\ \mathrm{times}}) \big).
		\end{eqnarray*}

		It is easy to see that
		$l_2(s\pi, sf)=- s[\pi, f]_{\G}$ induces a  differential on $\frakm=\Hom(\T(A),A)$. Denote it by $\partial_\Alg=[\pi,-]_{\G}$. $(\frakm,\partial_\Alg)$ is just the cochain complex of $A$ in Definition \ref{cohomology of algebra}

		Let us compute $l_1(sf)+l_2(s\pi, g)+l_3(s\pi, T, g)
		+l_2(sf, T)+ \sum\limits_{k=2}^{n}\frac{1}{k!}l_{k+1}(sf, \underbrace{T,\cdots,T}_{k\ \mathrm{times}})$.
		
		For   $x_1,\dots,x_{n}\in A$, note that $(l_2(s\pi, g)+l_3(s\pi, T, g))(x_1, \dots, x_{n})$
		induces a  differential on $\fraka=\Hom(\T(A),A)$ just as in Proposition-Definition \ref{cohomology of operator}. Denote this differential by $\partial_\RB$. So $(\fraka,\partial_\RB)$ is just the cochain complex of $T$.
		
		On the other hand,
		\begin{eqnarray*}
			&&l_1(sf)+l_2(sf, T)+\sum\limits_{k=2}^{n}\frac{1}{k!}l_{k+1}(sf, \underbrace{T,\cdots,T}_{k\ \mathrm{times}})(x_1, \dots, x_{n})\\
			&=&-A_{n-1}f-A_{n-2}f\bar{\circ} T-B_{n}T\bar{\circ} f\\
			&&-\sum_{k=2}^{n} \sum\limits_{\substack{\sigma\in \rmS_{k}\\a_j+m_{\sigma(j)}+1\le a_{j+1}}} A_{n-k-1}((f\bar{\circ}_{a_1}T)\bar{\circ}_{a_2}T)\bar{\circ}_{a_3}\cdots\big)\bar{\circ}_{a_{k}}T(x_1, \dots, x_{n})\\
			&&-\sum_{k=2}^{n} B_{n-k+1} \sum\limits_{\substack{\sigma\in \rmS_{k}\\a_j+m_{\sigma(j)}+1\le a_{j+1}}} 
			T\bar{\circ}\big(\cdots ((f\bar{\circ}_{a_1}T)\bar{\circ}_{a_2}\cdots)\bar{\circ}_{a_{k-1}}T\big) (x_1, \dots, x_{n})\\
			&=&-\sum\limits_{\substack{0\le i_1\le \cdots \le i_{k}\le n \\0\le k \le n}}A_{n-k-1}f(x_1, \cdots, T(x_{i_1}),\cdots, T(x_{i_{k}}),\cdots,x_n )\\
			&&-\sum\limits_{\substack{0\le i_1\le \cdots \le i_{k}\le n\\0\le k\le n-1}} B_{n-k+1}T\circ f(x_1, \cdots, T(x_{i_1}),\cdots, T(x_{i_{k}}),\cdots,x_n )\\
		\end{eqnarray*}
		It  induces a linear map $\frakm\to  \fraka$, denote by $\delta$.
		Hence we have
		$$ l_1^{(s\pi, T)}(sf, g) =\big(-s\partial_\Alg^{n}(f), \partial_\RB^{n-1}(g)+\delta^n(f)\big),$$ and $\delta$ is a cochain map.
	\end{proof}
	
	From the above proof, we could see that:
	\begin{defn}\label{RBA homology}
		With the above notation.
			\textbf{The cochain complex of the extended Rota-Baxter algebra $(A,T)$} is defined to be $(\rmC^*_{\RBA}(A)=s\frakm\oplus \fraka,\partial_\RBA)$ where $\partial_\RBA^{n}(sf,g)=\big(s\partial_\Alg^{n}(f), -\partial_\RB^{n-1}(g)-\delta^n(f)\big)$. It's induced by the mapping cone of the cochain map $\delta$ from the cochain complex of the algebra $A$ to the cochain complex of operator $T$.  The corresponding cohomology group is called \textbf{the cohomology of the extended Rota-Baxter algebra $(A,T)$}. 
	
	\end{defn}

	We generalize the above definition to the case of coefficients in an arbitrary extended Rota-Baxter bimodule.
	
	Let $(M,T_M)$ be a bimodule of the extended Rota-Baxter algebra $(A,T)$ with weight $(\mu,\lambda)$.

	Consider the trivial extension  $ A \ltimes M$ of the extended Rota-Baxter algebra $(A,T)$ by the bimodule $(M,T_M)$ in Proposition~\ref{Prop: trivial extension of extended Rota-Baxter bimodule}, then we get the complex $( \rmC^*_{\RBA}(A \ltimes M),\partial_{\RBA}^*)$ by  Definition \ref{RBA homology}. 
	
	\begin{prop-def} \label{RBA homology on M}
		\textbf{The cochain complex of the extended Rota-Baxter algebra $(A,T)$ with coefficients in the bimodule $(M,T_M)$},  denoted by $(\rmC_{\RBA}^*(A, M),\partial_{\RBA}^*)$, is the subcomplex of
		$( \rmC^*_{\RBA}(A \ltimes M, (A \ltimes M)_\ad),\partial_{\RBA}^*)$ given by the natural injection
		from $$\rmC_{\RBA}^n(A, M)=\mathrm{Hom}( A^{\otimes n},M)\oplus \mathrm{Hom}( A^{\otimes n-1},M)$$ to $$\rmC^n_{\RBA}(A \ltimes M, (A \ltimes M)_\ad)=\mathrm{Hom}( (A\oplus M)^{\otimes n},A\oplus M)\oplus \mathrm{Hom}( (A\oplus M)^{\otimes n-1},A\oplus M).$$ The corresponding cohomology group is called \textbf{the cohomology of the extended Rota-Baxter algebra $(A,T)$ with coefficients in the extended Rota-Baxter bimodule $(M,T_M)$}. 
	\end{prop-def}
	
	\begin{rmk}
		Now we provide a more detailed description of the cochain complex of extended Rota-Baxter algebras. Likewise, we can generalize the cohomology of algebras and extended Rota-Baxter operators to the case of coefficients in an arbitrary bimodule.
		
		Let $(A,\pi)$ be an associative algebra and  $M$ be a  bimodule over $A$. \textbf{The cochain complex of $A$ with coefficients in $M$} is defined to be $(\rmC^*_{\Alg}(A,M):=\mathrm{Hom}(\T(A),M),\partial_{\Alg})$ where the coboundary operator $\partial_{\Alg}$ is given as follows. The corresponding cohomology is called \textbf{the cohomology of $A$ with coefficients in $M$}. We could see that this is just the Hochschild cochain complex and Hochschild cohomology of $A$ with coefficients in $M$. \\
		
		The coboundary operator $$\partial_{\Alg}^n: \mathrm{Hom}(A^{\otimes n},M)\longrightarrow  \mathrm{Hom}(A^{\otimes n+1},M), n\geq 0$$ is given by
		\[\begin{split}
			\partial_{\Alg}^n (f)(x_1,\dots,x_{n+1})&=(-1)^n x_1 f(x_2,\dots, x_n)+f(x_1,\dots x_{n-1}) x_n\\
			&+\sum_{i=1}^{n-1} (-1)^{i+n}f( x_1,\dots,x_{i-1},x_i x_{i+1}, \dots,x_n),
		\end{split}\]
		for all $f\in \mathrm{Hom}(A^{\otimes n},M),~x_1,\dots, x_{n+1}\in A$.

		Consider an extended Rota-Baxter algebra $(A,T)$ and $(M, T_M)$ is an extended Rota-Baxter bimodule over $(A,T)$. \textbf{The cochain complex of the extended Rota-Baxter operator  $T$ with coefficients in $M$} $(\rmC^*_{\RB}(A,M):=\mathrm{Hom}(\T(A),M),\partial_{\RB})$ is defined to be the Hochschild cochain complex   of $A_\star$ with coefficient in $_\rhd M_\lhd$(See Proposition  \ref{Prop: new RB algebra} and \ref{descending property}). The corresponding  cohomology group is called \textbf{the cohomology of the extended Rota-Baxter operator  $T$ with coefficients in $M$}.
		
		There exists a cochain map $\delta: \rmC^*_{\Alg}(A,M)\to \rmC^*_{\RB}(A,M)$ defined by
		\begin{align*}
			&\delta(f)(x_1,\cdots,x_n)\\
			=&-\sum\limits_{\substack{0\le i_1\le \cdots \le i_{k}\le n \\0\le k \le n}}A_{n-k-1}f(x_1, \cdots, T(x_{i_1}),\cdots, T(x_{i_{k}}),\cdots,x_n )\\
			&-\sum\limits_{\substack{0\le i_1\le \cdots \le i_{k}\\0\le k\le n-1}} B_{n-k+1}T\circ f(x_1, \cdots, T(x_{i_1}),\cdots, T(x_{i_{k}}),\cdots,x_n ) ,
		\end{align*}
		for $f\in \mathrm{Hom}(A^{\otimes n},M)$ and $x_1,\dots,x_n \in A$.
		
		Then the  coboundary operator of  cochain complex of the extended Rota-Baxter algebra $(A,T)$ with coefficients in the  $(M,T_M)$ is just $\partial_{\RBA}^{n}(sf,g)=\big(s\partial_\Alg^{n}(f), -\partial_\RB^{n-1}(g)-\delta^n(f)\big)$ for $f\in \mathrm{Hom}( A^{\otimes n},M), g\in  \mathrm{Hom}( A^{\otimes n-1},M)$.
		
		 We compute $n$-cocycles of $\rmC_{\RBA}^*(A,M)$ for small $n$.

		For all $(f,m)\in\Hom (A,M)\oplus M$, $\partial^1_{\RBA} (f,m)=0$ if and only if $\partial^1_\Alg  f=0$ and
		\begin{align}\label{1-cocycle}
		T(x) m-T_M(xm)-mT(x)+T_M(mx)=f(T(x))-T_M(f(x)),\quad \forall x\in A.
     	\end{align}
		
		For all $(f,g)\in\Hom (A\o A,M)\oplus\Hom (A,M) $,  $\partial^2_{\RBA}(f,g)=0$ if and only if $\partial^2_\Alg f=0,$ and
		\begin{align}\label{2-cocycle}
			&-T(x) g(y)+T_M(x g(y))-  g(x)T(y)+T_M(g(x)y )+g(x \star y)\\
			= & -\lambda f(x,y)+f(T(x),T(y))-\mu T_Mf(x,y)-T_M(f(T(x),y))-T_M(f(x,T(y))),\notag
		\end{align}
		for all $x,y\in A.$
	\end{rmk}
	 
    \bigskip
	\section{Deformation theory, abelian extensions and homotopy analogs}\label{sec 6}
	The content of this section is as one would expect of a good cohomology theory and of a good controlling algebra of extended Rota-Baxter algebras.

	\subsection{Formal deformations of extended Rota-Baxter  algebras}\label{subsec 6.1}\
	
	In this subsection, we will study formal deformations of extended Rota-Baxter algebras and interpret  them  via the   low-degree   cohomology groups  of extended Rota-Baxter algebras defined in  Proposition-Definition \ref{RBA homology on M}.

	Let $(A, T)$ be an extended Rota-Baxter algebra of weight $(\mu,\lambda)$.   Consider a 1-parameter family
	\[\mu_t=\sum_{i=0}^\infty \mu_it^i, \ \mu_i\in \mathrm{Hom}( A\o A,A),\quad  T_t=\sum_{i=0}^\infty T_it^i,  \ T_i\in \mathrm{Hom}(A,A).\]
	
	\begin{defn}
		\textbf{A  1-parameter formal deformation} of the   extended Rota-Baxter algebra $(A, \pi,T)$ is a pair $(\mu_t,T_t)$ which endows the free $\bfk[[t]]$-module $A[[t]]$ with an extended  Rota-Baxter algebra structure over $\bfk[[t]]$ such that $T_0=T$ and $\mu_0=\pi$.
	\end{defn}
	
	The power series $\mu_t$ and $ T_t$ determine a  1-parameter formal deformation of the extended Rota-Baxter algebra $(A,T)$ if and only if for any $a,b,c\in A$, the following equations hold :
	\begin{eqnarray*}
		\mu_t(a\ot \mu_t(b\ot c))&=&\mu_t(\mu_t(a\ot b)\ot c),\\
		\mu_t(T_t(a)\ot T_t(b))&=& T_t\Big(\mu_t(a\ot T_t(b))+\mu_t(T_t(a)\ot b)+\mu \mu_t(a\ot b)\Big)+\lambda \mu_t(a\ot b).
	\end{eqnarray*}
	By expanding these equations and comparing the coefficient of $t^n$, we obtain  that $\{\mu_i\}_{i\geqslant0}$ and $\{T_i\}_{i\geqslant0}$ have to  satisfy: for any $n\geqslant 0$,
	\begin{equation}\label{Eq: deform eq for  products in RBA}
		\sum_{i=0}^n\mu_i\circ(\mu_{n-i}\ot \id)=\sum_{i=0}^n\mu_i\circ(\id\ot \mu_{n-i}),\end{equation}
	\begin{equation}\label{Eq: Deform RB operator in RBA} \begin{array}{rcl}
			\sum_{i+j+k=n\atop i, j, k\geqslant 0}	\mu_{i}\circ(T_j\ot T_{k})&=&\sum_{i+j+k=n\atop i, j, k\geqslant 0} T_{i}\circ \mu_j\circ (\id\ot T_{k})\\
			&  &+\sum_{i+j+k=n\atop i, j, k\geqslant 0} T_{i}\circ\mu_j\circ (T_{k}\ot \id)+\mu\sum_{i+j=n\atop i, j \geqslant 0}T_i\circ\mu_{j}+\lambda \mu_n.
	\end{array}\end{equation}
	Obviously, when $n=0$, it is just the extended Rota-Baxter algebra $(A,T)$.
	
	\smallskip
	
	\begin{prop}\label{Prop: Infinitesimal is 2-cocycle}
		Let $(A[[t]],\mu_t,T_t)$ be a  1-parameter formal deformation of the
		extended Rota-Baxter algebra $(A,T)$ of weight $(\mu,\lambda)$. Then
		$(\mu_1,T_1)$ is a 2-cocycle in the cochain complex
		$\rmC_{\RBA}^\bullet(A)$.
	\end{prop}
	\begin{proof} When $n=1$,   Equations~(\ref{Eq: deform eq for  products in RBA}) and (\ref{Eq: Deform RB operator in RBA})  become
		$$\mu_1\circ(\pi\ot \id)+\pi\circ(\mu_1\ot \id)=\mu_1\circ(\id\ot \pi)+\pi\circ (\id\ot \mu_1)$$
		and
		$$\begin{array}{cl}
			&\mu_1 (T\ot T)-\{T\circ\mu_1\circ(\id\ot T)+T\circ\mu_1\circ(T\ot \id)+\mu T\circ \mu_1+\lambda \mu_1\}\\
			=&T\circ \pi\circ(\id\ot T_1)+T\circ \pi\circ(T_1\ot\id)
			+T_1\circ \pi\circ(\id\ot T)+T_1\circ \pi\circ(T\ot \id)+\mu T_1\circ \pi\\
			&-\pi\circ(T_1\ot T)-\pi\circ(T\ot T_1),
		\end{array}$$
		Note that  the first equation is exactly $\partial_{\Alg}^2(\mu_1)=0\in \C^\bullet_{\Alg}(A)$ and that  second equation is exactly   \[\delta^2(\mu_1)=-\partial_{\RB}^1(T_1) \in \C^\bullet_{\RB}(A).\]
		So $(\mu_1,T_1)$ is a 2-cocycle in $\C^\bullet_{\RBA}(A)$, see Eq. (\ref{2-cocycle}).
	\end{proof}
	
	\smallskip
	
	\begin{defn} The 2-cocycle $(\mu_1,T_1)$ is called the \textbf{infinitesimal} of the 1-parameter formal deformation $(A[[t]],\mu_t,T_t)$ of the extended Rota-Baxter algebra $(A,T)$.
	\end{defn}

	In general, we can rewrite Equations~(\ref{Eq: deform eq for  products in RBA}) and (\ref{Eq: Deform RB operator in RBA}) as
	\begin{eqnarray} \label{Eq: general formal of deform product in RBA} \partial_{\Alg}^2(\mu_n) = \frac{1}{2}\sum_{i=1}^{n-1} [\mu_i, \mu_{n-i}]_{\G},\end{eqnarray}
	\begin{equation} \label{Eq: general formal of deform RBO in RBA}
		\begin{array}{rcl}\partial_{\RB}^{1}(T_n)   +\delta^2(\mu_n)&=& \sum_{i+j+k=n\atop 0 \leqslant i, j, k\leqslant n-1}	\mu_{i}\circ(T_j\ot T_{k})-\sum_{i+j+k=n\atop 0 \leqslant i, j, k\leqslant n-1} T_{i}\circ \mu_j\circ (\id\ot T_{k})\\ &&-\sum_{i+j+k=n\atop 0 \leqslant i, j, k\leqslant n-1} T_{i}\circ\mu_j\circ (T_{k}\ot \id)-\sum_{i+j=n\atop 0 \leqslant i, j\leqslant n-1}T_i\circ\mu_{j}.
	\end{array}\end{equation}
	
	\smallskip
	\begin{defn}
		Let $(A[[t]],\mu_t,T_t)$ and $(A[[t]],\mu_t',T_t')$ be two 1-parameter formal deformations of the extended Rota-Baxter algebra $(A,T)$.\textbf{ A formal isomorphism} from $(A[[t]],\mu_t',T_t')$ to $(A[[t]], \mu_t, T_t)$ is a power series $\psi_t=\sum_{i=0}\psi_it^i: A[[t]]\rightarrow A[[t]]$, where $\psi_i: A\rightarrow A$ are linear maps with $\psi_0=\id_A$, such that
		\begin{eqnarray}\label{Eq: equivalent deformations}\psi_t\circ \mu_t' &=& \mu_t\circ (\psi_t\ot \psi_t),\\
			\psi_t\circ T_t'&=&T_t\circ\psi_t. \label{Eq: equivalent deformations2}
		\end{eqnarray}
		In this case, we say that 1-parameter formal deformations $(A[[t]], \mu_t,T_t)$ and
		$(A[[t]],\mu_t',T_t')$ are  equivalent.
	\end{defn}

	\smallskip

	Given an extended Rota-Baxter algebra $(A,T)$, the power series $\mu_t,T_t$
	with $\mu_i=0, T_i=0$ for $i\ge 1$ make
	$(A[[t]],\mu_t,T_t)$ into a $1$-parameter formal deformation of
	$(A,T)$. The formal deformations equivalent to this one are called \textbf{trivial}.
	\smallskip
	
	\begin{thm}
		The infinitesimals of two equivalent 1-parameter formal deformations of $(A,T)$ are in the same cohomology class in $\rmH^2_{\RBA}(A)$.
	\end{thm}
	
	\begin{proof} Let $\psi_t:(A[[t]],\mu_t',T_t')\rightarrow (A[[t]],\mu_t,T_t)$ be a formal isomorphism.
		Expanding the identities and collecting coefficients of $t$, we get from Equations~(\ref{Eq: equivalent deformations}) and (\ref{Eq: equivalent deformations2}):
		\begin{eqnarray*}
			\mu_1'&=&\mu_1+\pi\circ(\id\ot \psi_1)-\psi_1\circ\pi+\pi\circ(\psi_1\ot \id),\\
			T_1'&=&T_1+T\circ\psi_1-\psi_1\circ T,
		\end{eqnarray*}
		that is, we have\[(\mu_1',T_1')-(\mu_1,T_1)=(\partial_{\Alg}^1(\psi_1), -\delta^1(\psi_1))=\partial_{\RBA}^1(\psi_1,0)\in  \C^2_{\RBA}(A),\] see Eq. (\ref{1-cocycle}).
	\end{proof}
	
	\smallskip
	
	\begin{defn}
		An extended Rota-Baxter algebra $(A,T)$ is said to be \textbf{rigid} if every 1-parameter formal deformation is trivial.
	\end{defn}
	
	\begin{thm}
		Let $(A,T)$ be an extended Rota-Baxter algebra of weight $(\mu,\lambda)$. If $\rmH^2_{\RBA}(A)=0$, then $(A,T)$ is rigid.
	\end{thm}
	
	\begin{proof}Let $(A[[t]], \mu_t, T_t)$ be a $1$-parameter formal deformation of $(A,T)$. By Proposition~\ref{Prop: Infinitesimal is 2-cocycle},
		$(\mu_1, T_1)$ is a $2$-cocycle. By $\rmH^2_{\RBA}(A)=0$, there exists a $1$-cochain $$(\psi_1', x) \in \C^1_\RBA(A)= C^1_{\Alg}(A)\oplus \Hom(\bfk, A)$$ such that
		$(\mu_1, T_1) = \partial_{\RBA}	(\psi_1', x), $
		that is, $\mu_1=\partial_{\Alg}^1(\psi_1')$ and $T_1=-\partial_{\RB}^0(x)-\delta^1(\psi_1')$. Let $\psi_1=\psi_1'+\partial_{\Alg}^0(x)$. Then
		$\mu_1= \partial_{\Alg}^1(\psi_1)$ and $T_1=-\delta^1(\psi_1)$, as it can be readily seen that $\delta^1(\partial_{\Alg}^0(x))=\partial_{\RB}^0(x)$.
		
		Setting $\psi_t = \Id_A -\psi_1t$, we have a deformation $(A[[t]], \overline{\mu}_t, \overline{T}_t)$, where
		$$\overline{\mu}_t=\psi_t^{-1}\circ \mu_t\circ (\psi_t\times \psi_t)$$
		and $$\overline{T}_t=\psi_t^{-1}\circ T_t\circ \psi_t.$$
		It can be easily verifed  that $\overline{\mu}_1=0, \overline{T}_1=0$. Then
		$$\begin{array}{rcl} \overline{\mu}_t&=& \pi+\overline{\mu}_2t^2+\cdots,\\
			\overline{T}_t&=& T+\overline{T}_2t^2+\cdots.\end{array}$$
		By Equations~(\ref{Eq: general formal of deform product in RBA}) and (\ref{Eq: general formal of deform RBO in RBA}), we see that $(\overline{\mu}_2,  \overline{T}_2)$ is still a $2$-cocycle, so by induction, we can show that
		$ (A[[t]], \mu_t , T_t) $ is equivalent to the trivial  deformation $(A[[t]],\pi,T).$
		Thus, $(A,T)$ is rigid.
		
	\end{proof}

	\medskip
	
	\subsection{Abelian extensions of extended Rota-Baxter algebras}\label{subsec 6.2}\
	
	In this subsection, we study abelian extensions of extended Rota-Baxter algebras and show that they are classified by the second cohomology group.

	Notice that a vector space $M$ together with a linear transformation $T_M:M\to M$ is naturally an extended Rota-Baxter algebra where the multiplication on $M$ is defined to be $uv=0$ for all $u,v\in M.$
	
	\begin{defn}
		\textbf{An   abelian extension}  of extended Rota-Baxter algebras is a short exact sequence of  morphisms of extended Rota-Baxter algebras
		\begin{eqnarray}\label{Eq: abelian extension} 0\to (M, T_M)\stackrel{i}{\to} (\hat{A}, \hat{T})\stackrel{p}{\to} (A, T)\to 0,
		\end{eqnarray}
		that is, there exists a commutative diagram:
		\[\begin{CD}
			0@>>> {M} @>i >> \hat{A} @>p >> A @>>>0\\
			@. @V {T_M} VV @V {\hat{T}} VV @V T VV @.\\
			0@>>> {M} @>i >> \hat{A} @>p >> A @>>>0,
		\end{CD}\]
		where the extended Rota-Baxter algebra $(M, T_M)$	satisfies  $uv=0$ for all $u,v\in M.$
		
		We will call $(\hat{A},\hat{T})$ an abelian extension of $(A,T)$ by $(M,T_M)$.
	\end{defn}

	\begin{defn}
		Let $(\hat{A}_1,\hat{T}_1)$ and $(\hat{A}_2,\hat{T}_2)$ be two abelian extensions of $(A,T)$ by $(M,T_M)$. They are said to be  \textbf{isomorphic}  if there exists an isomorphism of extended Rota-Baxter algebras $\zeta:(\hat{A}_1,\hat{T}_1)\rar (\hat{A}_2,\hat{T}_2)$ such that the following commutative diagram holds:
		\begin{eqnarray}\label{Eq: isom of abelian extension}\begin{CD}
				0@>>> {(M,T_M)} @>i >> (\hat{A}_1,{\hat{T}_1}) @>p >> (A,T) @>>>0\\
				@. @| @V \zeta VV @| @.\\
				0@>>> {(M,T_M)} @>i >> (\hat{A}_2,{\hat{T}_2}) @>p >> (A,T) @>>>0.
		\end{CD}\end{eqnarray}
	\end{defn}
	
	A   section of an abelian extension $(\hat{A},{\hat{T}})$ of $(A,T)$ by $(M,T_M)$ is a linear map $s:A\rar \hat{A}$ such that $p\bar{\circ} s=\Id_A$. We identify $M$ with $i(M).$
	
	\bigskip
	
	Let    $(\hat{A},\hat{T})$ be  an abelian extension of $(A,T)$ by $(M,T_M)$ having the form \eqref{Eq: abelian extension}. Choose a section $s:A\rar \hat{A}$. We   define
	$$
	am:=s(a)m,\quad ma:=ms(a), \quad \forall a\in A, m\in M.
	$$
	\begin{prop}\label{Prop: new RB bimodules from abelian extensions}
		With the above notation, $(M, T_M)$ is an extended Rota-Baxter  bimodule over $(A,T)$.
	\end{prop}
	\begin{proof}
		For any $a,b\in A,\,m\in M$, since $s(ab)-s(a)s(b)\in M$ implies $s(ab)m=s(a)s(b)m$, we have
		\[ (ab)  m=s(ab)m=s(a)s(b)m=a(bm).\]
		Hence,  this gives a left $A$-module structure and the case of right module structure is similar.

		Moreover, ${\hat{T}}(s(a))-s(T(a))\in M$ means that  ${\hat{T}}(s(a))m=s(T(a))m$. Thus we have
		\begin{align*}
			T(a)T_M(m)&=s(T(a))T_M(m)\\
			&=\hat{T}(s(a))T_M(m)\\
			&=\hat{T}(\hat{T}(s(a))m+s(a)T_M(m)+\mu s(a)m)+\lambda s(a)m \\
			&=T_M(T(a)m+aT_M(m)+\mu am)+\lambda am
		\end{align*}
		It is similar to see $T_M(m)T(a)=T_M(T_M(m)a+mT(a)+\mu ma) +\lambda ma$.
		
		Hence, $(M, T_M)$ is an extended  Rota-Baxter  bimodule over $(A,T)$.
	\end{proof}
	
	We  further  define linear maps $\psi:A\ot A\rar M$ and $\chi:A\rar M$ respectively by
	\begin{align*}
		\psi(a\ot b)&=s(a)s(b)-s(ab),\quad\forall a,b\in A,\\
		\chi(a)&={\hat{T}}(s(a))-s(T(a)),\quad\forall a\in A.
	\end{align*}

	\begin{prop}\label{psi and chi}
		The pair
		$(\psi,\chi)$ is a 2-cocycle  of  the extended Rota-Baxter algebra $(A,T)$ with  coefficients  in the extended Rota-Baxter bimodule $(M,T_M)$ introduced in  Proposition~\ref{Prop: new RB bimodules from abelian extensions}.
	\end{prop}

	\begin{proof}
		
		To see $(\psi,\chi)$ is a  2-cocycle, it suffices to check the identities in Eq. (\ref{2-cocycle}).
		
		For the first identity, 
		\begin{align*}
			&a\psi(b\ot c)-\psi(ab\ot c)+\psi(a\ot bc)-\psi(a\ot b)c\\
			=&a(s(b)s(c)-s(bc))-s(ab)s(c)+s(abc)+s(a)s(bc)-s(abc)-(s(a)s(b)-s(ab))c\\
			=&s(a)s(b)s(c)-s(a)s(bc)-s(ab)s(c)+s(a)s(bc)-s(a)s(b)s(c)+s(ab)s(c)\\
			=&0
		\end{align*}
		Thus $\partial^2_\Alg (\psi)=0$.
		
		For the second identity, 
		\begin{align*}
			&T(a)\chi(b)+\chi(a)T(b)+\psi(T(a)\ot T(b))\\
			=&s(T(a))({\hat{T}}(s(b))-s(T(b)))+({\hat{T}}(s(a))-s(T(a)))s(T(b))+s(T(a))s(T(b))-s(T(a)T(b))\\
			=&\hat{T}(s(a))({\hat{T}}(s(b))-s(T(b)))+({\hat{T}}(s(a))-s(T(a)))\hat{T}(s(b))+s(T(a))s(T(b))-s(T(a)T(b))\\
			=&2\hat{T}(s(a))\hat{T}(s(b))-\hat{T}(s(a))s(T(b))-s(T(a))\hat{T}(s(b))+s(T(a))s(T(b))-s(T(a)T(b))\\
			=&2\hat{T}(s(a))\hat{T}(s(b))-\hat{T}(s(a))s(T(b))-\hat{T}(s(a))(\hat{T}(s(b))-s(T(b)))-s(T(a)T(b))\\
			=&\hat{T}(s(a))\hat{T}(s(b))-s(T(a)T(b))
		\end{align*}
		
		\begin{align*}
			&T_M(\chi(a)b)+T_M(\psi(T(a)\ot b))+\chi(T(a)b)+T_M(a\chi(b))+T_M(\psi(a\ot T(b)))+\chi(aT(b))\\
			&+\mu T_M(\psi(a\ot b))+\lambda \psi(a\ot b)+\mu \chi(ab)\\
			=&\hat{T}((\hat{T}(s(a))-s(T(a)))s(b))+\hat{T}(s(T(a))s(b)-s(T(a)b))+{\hat{T}}(s(T(a)b))-s(T(T(a)b))\\
			&+\hat{T}(s(a)({\hat{T}}(s(b))-s(T(b))))+\hat{T}(s(a)s(T(b))-s(aT(b)))+{\hat{T}}(s(aT(b)))-s(T(aT(b)))\\
			&+\mu \hat{T}(s(a)s(b)-s(ab))+\lambda(s(a)s(b)-s(ab))+\mu ({\hat{T}}(s(ab))-s(T(ab)))\\
			=&\hat{T}(\hat{T}(s(a))s(b)+s(a)\hat{T}(s(b))+\mu s(a)s(b))+\lambda s(a)s(b)\\
			&-s(T(aT(b)+T(a)b+\mu ab))-\lambda s(ab)\\
			=&T(a)\chi(b)+\chi(a)T(b)+\psi(T(a)\ot T(b))
		\end{align*}
		
		Thus \[ \partial_{\RB}^1(\chi)+\delta^2(\psi)=0.\]
		
	\end{proof}

	The choice of the section $s$ in fact determines a splitting
	$$\xymatrix{0\ar[r]&  M\ar@<1ex>[r]^{i} &\hat{A}\ar@<1ex>[r]^{p} \ar@<1ex>[l]^{t}& A \ar@<1ex>[l]^{s} \ar[r] & 0}$$
	subject to $t\circ i=\Id_M, t\circ s=0$ and $ it+sp=\Id_{\hat{A}}$.
	Then there is an induced isomorphism of vector spaces
	$$\left(\begin{array}{cc} p& t\end{array}\right): \hat{A}\cong   A\oplus M: \left(\begin{array}{c} s\\ i\end{array}\right).$$
	We now  transfer the extended Rota-Baxter algebra structure on $\hat{A}$ to $A\oplus M$ via this isomorphism. 	Since the following  proposition is obtained by straightforward  computation, we omit the details for brevity.
	\begin{prop}\label{product and operator}
		It is direct to verify that this  endows $A\oplus M$ with an associative product $\cdot_\psi$ and an extended Rota-Baxter operator $T_\chi$ defined by
		\begin{align}
			\label{eq:mul}(a,m)\cdot_\psi(b,n)&=(ab,an+mb+\psi(a,b)),\,\forall a,b\in A,\,m,n\in M,\\
			\label{eq:dif}T_\chi(a,m)&=(T(a),\chi(a)+T_M(m)),\,\forall a\in A,\,m\in M.
		\end{align}
		Moreover, we get an abelian extension
		$$0\to (M, T_M)\stackrel{\left(\begin{array}{cc} 0& 1\end{array}\right) }{\to} (A\oplus M, T_\chi)\stackrel{\left(\begin{array}{c} 1\\ 0\end{array}\right)}{\to} (A, T)\to 0$$
		which is easily seen to be  isomorphic to the original one \eqref{Eq: abelian extension}.
	\end{prop}

    By Propositions \ref{psi and chi} and \ref{product and operator}, we have:
   \begin{cor}\label{prop:2-cocycle}
   	The triple $(A\oplus M,\cdot_\psi,T_\chi)$ is an extended Rota-Baxter algebra   if and only if
   	$(\psi,\chi)$ is a 2-cocycle  of the extended Rota-Baxter algebra $(A,T)$ with  coefficients  in $(M,T_M)$.
   \end{cor}

	\medskip
	
	Now we investigate the influence of different choices of   sections.

	\begin{prop}\ \label{prop: different sections give}
		\begin{itemize}
			\item[(i)] Different choices of the section $s$ give the same  extended Rota-Baxter bimodule structures on $(M, T_M)$.
			
			\item[(ii)]   The cohomology class of $(\psi,\chi)$ does not depend on the choice of sections.
			
		\end{itemize}
		
	\end{prop}
	\begin{proof}Let $s_1$ and $s_2$ be two distinct sections of $p$.
		We define $\gamma:A\rar M$ by $\gamma(a)=s_1(a)-s_2(a)$.
		
		Since $uv=0$ for all $u,v\in M$,
		$$s_1(a)m= s_2(a)m+\gamma(a)m=s_2(a)m.$$ So different choices of the section $s$ give the same  extended Rota-Baxter bimodule structures on $(M, T_M)$;

		We   show that the cohomology class of $(\psi,\chi)$ does not depend on the choice of sections.   We have
		\begin{align*}
			\psi_1(a,b)&=s_1(a)s_1(b)-s_1(ab)\\
			&=(s_2(a)+\gamma(a))(s_2(b)+\gamma(b))-(s_2(ab)+\gamma(ab))\\
			&=(s_2(a)s_2(b)-s_2(ab))+s_2(a)\gamma(b)+\gamma(a)s_2(b)-\gamma(ab)\\
			&=(s_2(a)s_2(b)-s_2(ab))+a\gamma(b)+\gamma(a)b-\gamma(ab)\\
			&=\psi_2(a,b)+\partial_{\Alg}(\gamma)(a,b)
		\end{align*}
		and
		\begin{align*}
			\chi_1(a)&={\hat{T}}(s_1(a))-s_1(T(a))\\
			&={\hat{T}}(s_2(a)+\gamma(a))-(s_2(T(a))+\gamma(T(a)))\\
			&=({\hat{T}}(s_2(a))-s_2(T(a)))+{\hat{T}}(\gamma(a))-\gamma(T(a))\\
			&=\chi_2(a)+T_M(\gamma(a))-\gamma(T(a))\\
			&=\chi_2(a)-\delta^1(\gamma)(a).
		\end{align*}
		That is, $(\psi_1,\chi_1)=(\psi_2,\chi_2)+\partial_{\RBA}(\gamma)$. Thus $(\psi_1,\chi_1)$ and $(\psi_2,\chi_2)$ form the same cohomology class  {in $\rmH_{\RBA}^2(A,M)$}.
		
	\end{proof}
	
	We show now the isomorphic abelian extensions give rise to the same cohomology class.
	\begin{prop}Let $M$ be a vector space and  $T_M\in\End_\bfk(M)$. Let $(M, T_M)$ be an extended Rota-Baxter algebra with trivial multiplication.
		Let $(A,T)$ be an extended Rota-Baxter algebra.
		Then any two isomorphic abelian extensions of extended Rota-Baxter algebra $(A, T)$ by  $(M, T_M)$  give rise to the same cohomology class  in $\rmH_{\RBA}^2(A,M)$.
	\end{prop}
	\begin{proof}
		Assume that $(\hat{A}_1,{\hat{T}_1})$ and $(\hat{A}_2,{\hat{T}_2})$ are two isomorphic abelian extensions of $(A,T)$ by $(M,T_M)$ as is given in \eqref{Eq: isom of abelian extension}. Let $s_1$ be a section of $(\hat{A}_1,{\hat{T}_1})$. As $p_2\circ\zeta=p_1$, we have
		\[p_2\circ(\zeta\circ s_1)=p_1\circ s_1=\Id_{A}.\]
		Therefore, $\zeta\circ s_1$ is a section of $(\hat{A}_2,{\hat{T}_2})$. Denote $s_2:=\zeta\circ s_1$. Since $\zeta$ is a homomorphism of extended Rota-Baxter  algebras such that $\zeta|_M=\Id_M$, $\zeta(am)=\zeta(s_1(a)m)=s_2(a)m=am$, so $\zeta|_M: M\to M$ is compatible with the induced extended Rota-Baxter bimodule structures of $M$.
		We have
		\begin{align*}
			\psi_2(a\ot b)&=s_2(a)s_2(b)-s_2(ab)=\zeta(s_1(a))\zeta(s_1(b))-\zeta(s_1(ab))\\
			&=\zeta(s_1(a)s_1(b)-s_1(ab))=\zeta(\psi_1(a,b))\\
			&=\psi_1(a,b)
		\end{align*}
		and
		\begin{align*}
			\chi_2(a)&={\hat{T}_2}(s_2(a))-s_2(T(a))={\hat{T}_2}(\zeta(s_1(a)))-\zeta(s_1(T(a)))\\
			&=\zeta({\hat{T}_1}(s_1(a))-s_1(T(a)))=\zeta(\chi_1(a))\\
			&=\chi_1(a).
		\end{align*}
		Consequently, two isomorphic abelian extensions give rise to the same element in {$\rmH_{\RBA}^2(A,M)$}.
	\end{proof}
	\bigskip

	Now we consider the reverse direction.

	\begin{prop}
		Two cohomologous $2$-cocycles give rise to isomorphic abelian extensions.
	\end{prop}
	\begin{proof}

		Given two 2-cocycles $(\psi_1,\chi_1)$ and $(\psi_2,\chi_2)$, we can construct two abelian extensions $(A\oplus M,\cdot_{\psi_1},T_{\chi_1})$ and  $(A\oplus M,\cdot_{\psi_2},T_{\chi_2})$ via Equations~\eqref{eq:mul} and \eqref{eq:dif}. If they represent the same cohomology  class {in $\rmH_{\RBA}^2(A,M)$}, then there exist two linear maps $\gamma_0:\textbf{k}\rightarrow M, \gamma_1:A\to M$ such that $$(\psi_1,\chi_1)=(\psi_2,\chi_2)+(\partial_{\Alg}^1(\gamma_1),-\delta^1(\gamma_1)-\partial_{\RB}^0(\gamma_0)).$$
		Notice that $\partial_{\RB}^0=\delta^1\bar{\circ}\partial_{\Alg}^0$. Define $\gamma: A\rightarrow M$ to be $\gamma_1+\partial_{\Alg}^0(\gamma_0)$. Then $\gamma$ satisfies
		\[(\psi_1,\chi_1)=(\psi_2,\chi_2)+(\partial_{\Alg}^1(\gamma),-\delta^1(\gamma)).\]
		
		Define $\zeta:A\oplus M\rar A\oplus M$ by
		\[\zeta(a,m):=(a, \gamma(a)+m).\]
		Then $\zeta$ is an isomorphism of these two abelian extensions $(A\oplus M,\cdot_{\psi_1},T_{\chi_1})$ and  $(A\oplus M,\cdot_{\psi_2},T_{\chi_2})$.
	\end{proof}
	
	Finally, by all the propositions above in this subsection, we have:
	\begin{prop}
		 There is a bijection between the isomorphism classes of  abelian extensions of $(A,T)$ by $(M,T_M)$ and the second cohomology group   ${\rmH}_{\RBA}^2(A,M)$.
	\end{prop}

\medskip

\subsection{Homotopy extended Rota-Baxter algebras}\ \label{subsec 6.3}

In this subsection, we will introduce the notion of homotopy extended Rota-Baxter algebras with weight.

   Let $A=\bigoplus\limits_{i\in \mathbb{Z}}A^i$ be a graded space.  
   \begin{defn}
   	 Consider the graded space $\Hom(\T(A),A)$. The \textbf{graded partial Gerstenhaber composition} of two operators  $f$ and $g$ on position $i$ is defined to be
   	 $$f\bar{\circ}_i g(x_1,\dots,x_{m+n+1}):=(-1)^{q(\sum_{j=1}^{i-1} |x_i|)} f(x_1,\dots,x_{i-1},g(x_{i},\dots,x_{i+n}),x_{i+n+1},\dots,x_{m+n+1})$$ for $f\in \mathrm{Hom}(A^{\otimes m+1},A)$ of degree $p$, $g\in \mathrm{Hom}(A^{\otimes n+1},A)$, homogeneous elements $x_1,\dots,x_{m+n+1}\in A$  and $1\le i \le m+1$.
   	
   	Keeping the same notation, the \textbf{graded Gerstenhaber composition} of $f$ and $g$ is given by $f\bar{\circ} g=\sum_{i=1}^{m+1} f\bar{\circ}_i g$ and the \textbf{graded Gerstenhaber bracket} $[-,-]_{\bf{G}}$ defined on $\mathrm{Hom}(\T(A),A) $ is given by $[f,g]_{\bf{G}}=f\bar{\circ} g-(-1)^{pq}g\bar{\circ} f$.
   \end{defn}
   
   \begin{prop}\
   	
   	\begin{enumerate}
   		\item Let $A$ be a vector space. Then $(s\mathrm{Hom}(\T(A),A),[\cdot,\cdot])$ forms a graded Lie algebra where we define $[sx,sy]:=(-1)^{|x|}s[x,y]_{\bf{G}}$ for any homogeneous elements $x\in \mathrm{Hom}(A^{\otimes p+1},A)$ and $y\in \mathrm{Hom}(A^{\otimes q+1},A)$.
   		\item An element  $s\pi=\{s\pi_i\}_{i\geqslant 1} \in s\mathrm{Hom}(\T(A),A)$ with $\pi_i\in \mathrm{Hom}(A^{\otimes i},A)$  is a Maurer-Cartan element   if and only if $(A,\pi)$ is an  $A_\infty [1]$-algebra, see Remark \ref{A infty}.
   	\end{enumerate}
   \end{prop}
   
   \begin{rmk}
   	It is well known that $(\mathrm{Hom}(\T(A),A),[\cdot,\cdot]_{\bf{G}})$ is an ordinary graded Lie algebra. It follows that the above proposition holds. 
   \end{rmk}

   Recall $\overline{\T}(A)=\bigoplus\limits_{n=1}^\infty(A)^{\ot n}$.
    Denote $\mathfrak{C}_{\Alg}(A)=\Hom(\overline{\T}(A),A)$ and
   $$ \mathfrak{C}_{\RB}(A)=\Hom(\overline{\T}(A),A).$$
    Set $$\mathfrak{C}_{\RBA}(A) =s\mathfrak{C}_{\Alg}(A)\oplus\mathfrak{C}_{\RB}(A).$$  Then    $\mathfrak{C}_{\RBA}(A)$ is  an $L_\infty[1]$-algebra. The $L_\infty$-structure is given by:
   
   \begin{thm}\label{thm: Linfinity for homotopy}
   	Keep the above notation and denote $t=\sqrt{\mu^2-4\lambda}$. There exists an $L_\infty[1]$-algebra structure on   ${\mathfrak{C}_{\RBA}}(A)$, where $l_i$ are given by 
   	$$l_1(sf)=-A_n f$$
   	$$
   	l_2(sf,sg)      =    (-1)^{|f|} s[f,g]_{\G}, $$
   	and for $i\geq 2$,
   	\begin{equation*}
   		\begin{aligned}
   			l_i(sf,\xi_1,\cdots,\xi_{i-1})  =&\sum\limits_{\substack{\sigma\in \rmS_{i-1}\\a_j+m_{\sigma(j)}+1\le a_{j+1}}}\epsilon(\sigma)\\ &
   			-A_{n-i+1} \big(\cdots ((f\bar{\circ}_{a_1}\xi_{\sigma(1)})\bar{\circ}_{a_2}\xi_{\sigma(2)})\bar{\circ}_{a_3}\cdots\big)\bar{\circ}_{a_{i-1}}\xi_{\sigma(i-1)}\\
   			&-\sum\limits_{\substack{\sigma\in \rmS_{i-1}\\a_j+m_{\sigma(j)}+1\le a_{j+1}, j\ge 2}} \epsilon(\sigma) (-1)^{|f||\xi_{\sigma(1)}|} \\ &
   			B_{n-i+3} \xi_{\sigma(1)}\bar{\circ}\big(\cdots ((f\bar{\circ}_{a_2}\xi_{\sigma(2)})\bar{\circ}_{a_3}\cdots)\bar{\circ}_{a_{i-1}}\xi_{\sigma(i-1)}\big),
   		\end{aligned}
   	\end{equation*}
   	
   	for homogeneous elements	$f\in\Hom(A^{\otimes n+1},A)\subseteq\mathfrak{C}_{\Alg}(A)$, $g\in\Hom(A^{\otimes m+1},A)\subseteq \mathfrak{C}_{\Alg}(A)$    	and $\xi_j\in\Hom(A^{\otimes m_j+1},A)\subseteq \mathfrak{C}_{\RB}(A)$, $1\leq j\leq i-1$, and all other components vanish.
   \end{thm}

   \begin{rmk}
   	The proof is  similar to Theorem \ref{thm: Linfinity for ass case}.
   \end{rmk}

   \begin{prop-def}\label{Def: homotopy RB algebras}
   	Let $A=\bigoplus\limits_{i\in \mathbb{Z}}A^i$ be a graded space. Then a homotopy extended Rota-Baxter algebra $(A,\pi=\{\pi_i\}_{i\geqslant 1}, T=\{T_i\}_{i\geqslant 1})$  of weight $(\mu,\lambda)$, where $\pi_i,T_i\in \Hom(A^{\otimes i},A)$ for each $i$, is defined to be such that $(s\pi,T)$ is a Maurer-Cartan element in the $L_\infty[1]$-algebra $\mathfrak{C}_{\RBA}(A)$. The explicit formulas are given as follows.
   \end{prop-def}

   \begin{proof}
   	$(s\pi,T)=(\{s\pi_i\}_{i\geqslant 1}, \{T_i\}_{i\geqslant 1}) \in\mathcal{MC}(\mathfrak{C}_{\RBA}(A))$ if and only if
   	$[\pi,\pi]_{\G}=0$, i.e. 
   	\begin{eqnarray}
   		\label{Eq: A infinity}   \sum_{i+j+k= n,\atop i,   k\geq 0, j\geq 1  }\pi_{i+1+k}\circ(\Id^{\ot i}\ot \pi_j\ot \Id^{\ot k})=0,
   	\end{eqnarray}
   	
   	and
   	\begin{equation*}
   		\begin{aligned}
   			0 &= 	l_1(s\pi)+\sum_{k=2}^\infty\frac{1}{(k-1)!}l_k(s\pi,\underbrace{T,\dots,T}_{(k-1)\ \mathrm{times}}) 
   		\end{aligned}
   	\end{equation*}
   	For arbitrary  $(x_1,x_2,\dots,x_n)\in A^{\otimes n}$, we have
   	
   		\begin{align}\label{Eq: T infinity}
   			0=&	-A_{n-1}\pi_n\\
   			&+\sum\limits_{\substack{\\a_j+m_{\sigma(j)}+1\le a_{j+1}\\m\ge k\ge 1, j_1+\dots+j_k+m-k=n}}
   			-A_{m-k-1} \big(\cdots ((\pi_m\bar{\circ}_{a_1}T_{j_1})\bar{\circ}_{a_2}T_{j_2})\bar{\circ}_{a_3}\cdots\big)\bar{\circ}_{a_{k}}T_{j_k}\notag\\
   			&+\sum\limits_{\substack{\\a_j+m_{\sigma(j)}+1\le a_{j+1}, j\ge 2\\m-1\ge k\ge 0, j_1+\dots+j_k+m-k=n}}  
   			-B_{m-k+1} T_{j_1}\big(\cdots ((\pi_m\bar{\circ}_{a_2}T_{j_2})\bar{\circ}_{a_3}\cdots)\bar{\circ}_{a_{k}}T_{j_k}\big).\notag
   		\end{align}
   
   \end{proof}

   \begin{rmk}\label{A infty}
   	In the same ways that $L_\infty[1]$-algebras correspond to $L_\infty$-algebras, we say that $(A, \pi)$ is an $A_\infty[1]$-algebra.  Moreover, $T=\{T_i\}_{i\geqslant 1}$ is a homotopy extended Rota-Baxter operator on $(A,\pi)$.
   \end{rmk}

   For $n=1,2$, Equation~(\ref{Eq: T infinity})   gives
   \begin{eqnarray}
   	\label{operator-differential}   \pi_1\circ T_1=T_1\circ \pi_1,\end{eqnarray}
   and \begin{eqnarray}
   	\label{rbo-homotopy}  &  \pi_2\circ (T_1\ot T_1)-T_1\circ \pi_2\circ (\Id\ot T_1)-T_1\circ \pi_2\circ (T_1\ot \Id)-\mu T_1\circ \pi_2-\lambda \pi_2 \\
   	\notag   &= -\pi_1\circ T_2+T_2\circ (\Id\ot \pi_1)+T_2\circ(\pi_1\ot \Id).
   \end{eqnarray}
   Equation~(\ref{operator-differential}) implies that $T_1: (A, \pi_1)\to (A, \pi_1)$ is a cochain map,  thus $T_1$ is well-defined on  $\rmH^\bullet(A, \pi_1)$;  Equation~(\ref{rbo-homotopy}) indicates that $T_1$ is an extended Rota-Baxter operator of weight $(\mu,\lambda)$ with respect to $\pi_2$ up to homotopy, whose  obstruction is just the operator $T_2$. As a consequence, $(\rmH^\bullet(A, \pi_1), \pi_2, T_1)$ is  an extended Rota-Baxter algebra.

\bigskip

\section{Relations}\label{sec 7}

Furthermore, we describe the intrinsic relationship among extended Rota-Baxter algebras, extended Rota-Baxter Lie algebras, and the eigenvalues of the extended Rota-Baxter operator.

\subsection{The relation between $L_\infty$-structure and the eigenvalue of extended Rota-Baxter operators}\label{subsec 7.3}

	\begin{defn}
	Let $T$ be an operator on  $\bfk$-algebra $A$. If there exists $r\in \bfk$ such that  $r$ is a solution of relation equation of $T$, then  $r$ is called an \textbf{eigenvalue}  of $T$ and $-r$ is an \textbf{opposite eigenvalue} of $T$.
    \end{defn}

   \begin{exam}
   	Let $(A,\pi,T)$ be an extended Rota-Baxter algebra of weight $(\mu,\lambda)$. The relation equation of $T$ is given by $\pi\circ(T\o T)-T\pi(\Id\o T)-T\pi(T\o \Id)-\mu T\circ \pi-\lambda\pi=0$. Then the eigenvalue $r$ is the root of quadratic equation \( -r^2 - \mu r - \lambda= 0 \) and then the opposite eigenvalue $r$ of an extended Rota-Baxter algebra of weight $(\mu,\lambda)$ is the root of quadratic equation \( r^2 - \mu r + \lambda = 0 \). Thus we have two  opposite eigenvalues \( r_1 = \frac{\mu + t}{2} \) and \( r_2 = \frac{\mu - t}{2} \) where we set $t=\sqrt{\mu^2-4\lambda}$.
   \end{exam}

   \begin{prop}
   	The general terms of recursive sequence whose characteristic
   	equation is associated with extended Rota-Baxter operators are related to the $L_\infty$-structure of
   	extended Rota-Baxter algebras.
   \end{prop}
   
   \begin{proof}
   	Consider the recursive sequence $A_n=\mu A_{n-1}-\lambda A_{n-2}$ with characteristic equation \( r^2 - \mu r + \lambda = 0 \).
   	
   	When \[ A_0 = 0 \]
   	\[ A_1 = \lambda \]
   	
   	we have $A_n=\frac{\lambda}{t}((\frac{\mu +t}{2})^{n}-(\frac{\mu -t}{2})^{n})$
   	
   	When \[ A_0 = 0 \]
   	\[ A_1 = 1 \]
   	
   	we have $B_n=\frac{1}{t}((\frac{\mu +t}{2})^{n}-(\frac{\mu -t}{2})^{n})$.
   	
   	Hence these are precisely  the coefficients of the $L_\infty$-structure for extended Rota-Baxter algebras.
   \end{proof}

\medskip
 
\subsection{The relation between extended Rota-Baxter algebras and extended Rota-Baxter Lie algebras} \label{subsec 7.2}

\begin{defn}
	An extended Rota-Baxter Lie algebra of weight $(\mu,\lambda)$ for $\mu,\lambda \in \bfk$ is a pair $(A,T)$ where $A$ is a Lie algebra with a linear map $T:A\to A$ such that $[T(a),T(b)]=T([T(a),b]+[a,T(b)]+\mu [a,b]) +\lambda [a,b]$ for all $a,b \in A$.
	
	A morphism of extended Rota-Baxter Lie algebras of weight $(\mu,\lambda)$ between $(A,T)$ and $(B,T')$ is a morphism of Lie algebras $f: A \to B$ such that $f\circ T= T'\circ f$.
\end{defn}
   
   \begin{defn}(\cite{NR67})
   	Let $A$ be a vector space. Consider the graded vector space $ \mathrm{Hom}(\bigwedge (A),A):=\bigoplus_{n=-1}^{\infty}\mathrm{Hom}(\wedge^{n+1}A,A)$ where an element $f\in\mathrm{Hom}(\wedge^{n+1}A,A)$ is said to have degree $n$. 
   	Define the  \textbf{Nijenhuis-Richardson bracket} $[\cdot,\cdot]_{\bf{NR}}$ on $\mathrm{Hom}(\bigwedge (A),A)$  as follows:
   	for any $P\in \mathrm{Hom}(\wedge^{p+1}A,A)$ and $Q\in \mathrm{Hom}(\wedge^{q+1}A,A)$,
   	$$[P,Q]_{\bf{NR}}=P\bar{\circ}' Q-(-1)^{pq}Q\bar{\circ}' P,$$
   	where $P\bar{\circ}' Q$ is defined by
   	\begin{align*}
   		(P \bar{\circ}' Q)(x_{1}, \ldots, x_{p+q+1})=\sum_{\sigma \in \Sh(q+1, p)}sgn(\sigma) P(Q(x_{\sigma(1)}, \ldots, x_{\sigma(q+1)}), x_{\sigma(q+2)}, \ldots,  x_{\sigma(p+q+1)}).
   	\end{align*}
   \end{defn}

    Now, we will show  the relation between extended Rota-Baxter algebras and extended Rota-Baxter  Lie algebras.
    
        First, we introduce the controlling algebra of extended Rota-Baxter  Lie algebras.
    	Let $A$ be a vector space.  Let $$\frakm':=\mathrm{Hom}(\wedge(A),A)\ \mathrm{and}\
    \fraka':=\Hom(\wedge(A),A).$$
     
     For homogeneous elements	$f\in\Hom(\wedge^{n}A,A)$, $\sigma \in \rmS_n$ and $x_1,\dots,x_n\in A$, we define $f\sigma^{-1}(x_1,\dots,x_n):=f(x_{\sigma(1)},\dots,x_{\sigma(n)})$.
     
    	\begin{thm}[Controlling algebra of extended Rota-Baxter  Lie algebras\cite{yj}]\label{thm: Linfinity for extend Lie}
    	Keep the above notation and denote $t=\sqrt{\mu^2-4\lambda}$, there exists an $L_\infty[1]$-algebra structure on  $s \mathfrak{m}'\oplus\fraka'$, where $l'_i$ are given by 
    	$$l'_1(sf)=-A_n f$$
    	$$
    	l'_2(sf,sg)      =    (-1)^{|f|} s[f,g]_{\NR}, $$
    	and for $i\geq 2$,
    	\begin{equation*}
    		\begin{aligned}
    			 l'_i(sf,\xi_1,\cdots,\xi_{i-1}) =&\sum\limits_{\tau\in\Sh(m_1+1,\dots,m_{i-1}+1,n+2-i)}sgn(\tau)\\ &
    			-A_{n-i+1}\big(f(\xi_1\otimes\dots\otimes\xi_{i-1}\otimes\Id^{\otimes n+2-i})\big)\tau^{-1}\\
    			&+\sum_{k=1}^{i-1}\sum\limits_{\tau\in\Sh(m_1+1,\dots,m_{i-1}+1,n+3-i,m_k)}sgn(\tau)(-1)^{(n+\sum\limits_{j=1 }^{k-1}m_j)m_k} \\ 
    			&-B_{n-i+3} \xi_k  \big(f(\xi_1\otimes\dots\otimes\xi_{i-1}\otimes\Id^{\otimes n+3-i})\o \Id^{\otimes m_k}\big)\tau^{-1}	
    		\end{aligned}
    	\end{equation*}
    	
    	For homogeneous elements	$f\in\Hom(\wedge^{n+1}A,A)\subseteq \frakm'$, $g\in\Hom(\wedge^{m+1}A,A)\subseteq \frakm'$ and $\xi_j\in\Hom(\wedge^{m_j+1}A,A)\subseteq \fraka'$, $1\leq j\leq i-1$, and all other components vanish.
    \end{thm}

    \begin{prop}\label{L infty morphism}
    	With the same underlying space $A$, we have the  $L_\infty[1]$-algebra $s \mathfrak{m}\oplus\fraka$  of extended Rota-Baxter algebras and  $L_\infty[1]$-algebra $s \mathfrak{m'}\oplus\fraka'$ of extended Rota-Baxter Lie algebras. Then there exists an $L_\infty[1]$-algebra morphism $\phi:(s \mathfrak{m}\oplus\fraka,\{l_i\})\to (s \mathfrak{m}'\oplus\fraka',\{l'_i\}) $  where $\phi$ is given by:\\
    	$sf \mapsto \sum_{\sigma \in \rmS_n} sgn(\sigma) sf\sigma^{-1},$ \\
    	$g\mapsto \sum_{\sigma \in \rmS_m} sgn(\sigma) g \sigma^{-1}$,\\
    	for homogeneous elements	$f\in\Hom(A^{\otimes n},A)\subseteq \frakm$, $g\in\Hom(A^{\otimes m},A)\subseteq \fraka$  
    \end{prop}

    \begin{proof}
    	
    	For homogeneous elements	$f\in\Hom(A^{\otimes n+1},A)\subseteq \frakm$, $g\in\Hom(A^{\otimes m+1},A)\subseteq \frakm$	and $\xi_j\in\Hom(A^{\otimes m_j+1},A)\subseteq \fraka$, $1\leq j\leq i-1$.
    	
    	It's obvious that $\phi:s\mathfrak{m}\oplus\fraka\to s \mathfrak{m}'\oplus\fraka'$ is truly a linear map.
    	So we only need to show $\phi(l_1(sf))=l'_1(\phi(sf))$, $\phi(l_2(sf,sg))=l'_2(\phi(sf),\phi(sg))$ and $\phi(l_n(sf,\xi_1,\dots,\xi_{n-1}))=l'_n(\phi(sf),\dots,\phi(\xi_{n-1}))$.
    	
    	It's trivial for case one.
    	
    	For case two, it's suffices to show $\phi(f\bar{\circ} g)=\phi(f)\bar{\circ}^{'}\phi(g)$
    	\begin{equation*}
    		\begin{aligned}
    			&\phi(f\bar{\circ} g)(x_1,\dots,x_{m+n+1})\\
    			=&\sum_{i=1}^{n+1} (-1)^{(i-1)m}\phi(f\bar{\circ}_i g) (x_1,\dots,x_{m+n+1})\\
    			=&\sum_{i=1}^{n+1}(-1)^{(i-1)m}\sum_{\sigma \in \rmS_{m+n+1}}sgn(\sigma) f(x_{\sigma(1)},\dots,x_{\sigma(i-1)},g(x_{\sigma(i)},\dots,x_{\sigma(i+n)}),x_{\sigma(i+1)},\dots,x_{\sigma(m+n+1)})\\
    		   \xlongequal{\mbox{lemma \ref{bijection permutations}}}&\sum_{i=1}^{m+1}(-1)^{(i-1)m}\sum_{\tau \in n+\rmS_{n+1},\sigma \in \rmS_{m+1},\theta\in \Sh(m+1,n)}sgn(\sigma)sgn(\tau)sgn(\theta)(-1)^{(n+1)(i-1)}\\
    		   &f(x_{\theta\tau(n+2)},\dots,x_{\theta\tau(n+i)},g(x_{\theta\sigma(1)},\dots,x_{\theta\sigma(1+n)}),x_{\theta\tau(n+1+i)},\dots,x_{\theta\tau(n+1+m)})\\
    		   =&\sum_{\theta\in \Sh(m+1,n)}\phi(f)(\phi(g)(x_{\theta(1)},\dots,x_{\theta(n+1)}),x_{\theta(n+2)},\dots,x_{\theta(m+n+1)})\\
    		   =&\phi(f)\bar{\circ}^{'}\phi(g)(x_1,\dots,x_{m+n+1}),
    		\end{aligned}
    	\end{equation*} 
    	where $\tau(n+1)=i$, the position of $g$.
    	
    	For case three, we have 
    	\begin{equation*}
    		\begin{aligned} 
    			&\sum\limits_{\substack{\sigma\in \rmS_{i-1}\\a_j+m_{\sigma(j)}+1\le a_{j+1}}}\epsilon(\sigma)\phi\big(\big(\cdots ((f\bar{\circ}_{a_1}\xi_{\sigma(1)})\bar{\circ}_{a_2}\xi_{\sigma(2)})\bar{\circ}_{a_3}\cdots\big)\bar{\circ}_{a_{i-1}}\xi_{\sigma(i-1)}\big)\\
    			=&\sum\limits_{\substack{\sigma\in \rmS_{i-1}\\a_j+m_{\sigma(j)}+1\le a_{j+1}}}\epsilon(\sigma)\sum_{\tau \in \rmS_{n+1+\sum_{j=1}^{i-1}m_j}}\big(\big(\cdots ((f\bar{\circ}_{a_1}\xi_{\sigma(1)})\bar{\circ}_{a_2}\xi_{\sigma(2)})\bar{\circ}_{a_3}\cdots\big)\bar{\circ}_{a_{i-1}}\xi_{\sigma(i-1)}\big)\tau^{-1}\\
    			&\xlongequal{\mbox{Lemma \ref{general bijection permutations} (i)}}\sum\limits_{\tau\in\Sh(m_1+1,\dots,m_{i-1}+1,n+2-i)}sgn(\tau)\\
    			&\big(\sum_{\tau_1 \in \rmS_{n+1}}(sgn(\tau_1)f)(\sum_{\sigma_1 \in \rmS_{m_1+1}}(sgn(\sigma_1)\xi_1)\sigma_1^{-1}\otimes\dots\otimes\sum_{\sigma_{i-1} \in \rmS_{m_{i-1}+1}}(sgn(\sigma_{i-1})\xi_{i-1})\sigma_{i-1}^{-1}\otimes\Id^{\otimes n+2-i})\tau_1^{-1}\big)\tau^{-1}\\
    			=&\sum\limits_{\tau\in\Sh(m_1+1,\dots,m_{i-1}+1,n+2-i)}sgn(\tau)\big(\phi(f)(\phi(\xi_1)\otimes\dots\otimes\phi(\xi_{i-1})\otimes\Id^{\otimes n+2-i})\big)\tau^{-1}
    	\end{aligned}
    \end{equation*}
    and similar for
    \begin{equation*}
    	\begin{aligned} 
    		&\sum\limits_{\substack{\sigma\in \rmS_{i-1}\\a_j+m_{\sigma(j)}+1\le a_{j+1}, j\ge 2}} \epsilon(\sigma) (-1)^{nm_{\sigma(1)}}  \phi(\xi_{\sigma(1)}\bar{\circ}_{a_1}\big(\cdots ((f\bar{\circ}_{a_2}\xi_{\sigma(2)})\bar{\circ}_{a_3}\cdots)\bar{\circ}_{a_{i-1}}\xi_{\sigma(i-1)}\big))\\
    		&=\sum_{k=1}^{i-1}\sum\limits_{\tau\in\Sh(m_1+1,\dots,m_{i-1}+1,n+3-i,m_k)}sgn(\tau)(-1)^{(n+\sum\limits_{j=1 }^{k-1}m_j)m_k}  \phi(\xi_k)  \big(\phi(f)(\phi(\xi_1)\otimes\dots\otimes\phi(\xi_{i-1})\otimes\Id^{\otimes n+3-i})\o \Id^{\otimes m_k}\big)\tau^{-1}	
    	\end{aligned}
    \end{equation*} by Lemma \ref{general bijection permutations} (ii).
    \end{proof}
     \begin{rmk}
     	The  following lemmas give a correspondence of two different presentations of an  iterative composite of $f,g$ and $\xi_i$.
     \end{rmk}
 
     Denote $[n]:=(1,2,\dots,n)$.

    \begin{lem}\label{bijection permutations}
    	There is a 'natural' bijection between the set  $[n]\times \rmS_{m+n-1}$ and the set of permutations $\rmS_{m}\times \rmS_{n}\times \Sh(m,n-1)$. 
    \end{lem}

    Generally, we have:
    
    \begin{lem}\ \label{general bijection permutations}
    	
    	\begin{enumerate}
    		\item[(i)]There is a 'natural' bijection between the set  $\rmS_{k}\times [\left(\begin{array}{cc} n\\ k \end{array}\right)]\times \rmS_{n-k+\sum_{i=1}^{k}m_i}$ and the set of permutations $\rmS_{n}\times \rmS_{m_1}\times\dots\times \rmS_{m_k}\times \Sh(m_1,\dots,m_k,n-k).$

    		\item[(ii)]	And there is a 'natural' bijection between the set  $[m_k]\times\rmS_{k-1}\times [\left(\begin{array}{cc} n\\ k-1 \end{array}\right)]\times \rmS_{n-k+\sum_{i=1}^{k}m_i}$ and the set of permutations $\rmS_{n}\times \rmS_{m_1}\times\dots\times S_{m_k}\times \Sh(m_1,\dots,m_{k-1},n-k+1,m_k-1).$
    	\end{enumerate}

    \end{lem}
    \begin{proof}
    	It's trivial that there is a bijection if we calculate the size of the sets, but we  need a 'natural' bijection here.
    	
    	For case (i). Replacing  $\rmS_{m_j}$ by $m_1+\dots+m_{j-1}+\rmS_{m_j}$ for $j\ge 1$, denoted by  $\rmS_{m_j}$ for $j\ge 1$. 
    	
    	Define the map $f: \rmS_{k}\times [\left(\begin{array}{cc} n\\ k \end{array}\right)]\times \rmS_{n-k+\sum_{i=1}^{k}m_i} \to \rmS_{n}\times \rmS_{m_1}\times\dots\times \rmS_{m_k}\times \Sh(m_1,\dots,m_k,n-k)$ is given by:
    	
    	For $\sigma\in \rmS_{k},a\le i_1<i_2<\dots<i_k\le n ,\tau\in \rmS_{n-k+\sum_{i=1}^{k}m_i}$ and $\tau_1\in \rmS_{n},\sigma_1\in  \rmS_{m_1},\dots,\sigma_k\in \rmS_{m_k},\theta\in \Sh(m_1,\dots,m_k,n-k)$, we have
    	
    	$(\sigma,(i_1,\dots,i_k),\tau) \mapsto(\tau_1,\sigma_1,\dots,\sigma_k,\theta)$\\  
    	where $\theta$ is a $(m_1,\dots,m_k,n-k)$-shuffle given by the reodering of $ I_1\times I_2\times \cdots \times (I-\cup_{j=1}^{k}I_j)$,\\
    	here $I_j:=(\tau(i_k),\dots,\tau(i_k+m_j-1))$ for $j=\sigma(k)$, $k\ge j\ge 1$, let $I'_j$ be the set of reordering of $I_j$  and $I:=(\tau(1),\dots,\tau(n-k+\sum_{j=1}^{k}m_j))$.
    	
    	$\sigma_j$ is the permutation such that $\sigma_j I'_j=I_j$.
    	
    	And $\tau_1$ is the permutation induced by $(\tau(1),\dots,\tau(i_1-1),\sigma(1)-k,\tau(i_1+m_{\sigma(1)}),\dots,\tau(i_2-1),\sigma(2)-k,\tau(i_2+m_{\sigma(2)}),\dots,\dots,\tau(i_k-1),\sigma(k)-k,\tau(i_k+m_{\sigma(k)}),\dots,\tau(n-k+\sum_{i=1}^{k}m_i))$.
    	
    	The same analysis for case (ii).

    \end{proof}

    Now we know that $(A,\pi,T)$ is an  extended Rota-Baxter  algebra if and only if   $(s\pi,T)$ is a Maurer-Cartan element of $(s \mathfrak{m}\oplus\fraka,\{l_i\})$ by Theorem \ref{Thm: MC elements in ex Linifnity}. And from Proposition \ref{relations}, we have the following proposition:
    
    \begin{prop}
    	It's easy to see that an extended Rota-Baxter algebra $(A,\pi,T)$ induces an extended Rota-Baxter Lie algebra on $A$, where the Lie bracket $[-,-]$ is given by the commutator. Moreover, there exists an induced cochain map from the cochain complex of the extended Rota-Baxter algebra $(A,\pi,T)$ to the cochain complex of the extended Rota-Baxter Lie algebra $(A,[-,-],T)$.
    \end{prop}

  \smallskip
    Now we consider the homotopy version.
    
   	Let $V=\bigoplus\limits_{i\in \mathbb{Z}}V^i$ be a graded space. 
   Recall $\overline{\rmS}(V)=\bigoplus\limits_{n=1}^\infty \rmS^n(V)$, denote $$\mathfrak{C'}_{\Lie}(V)=\Hom(\overline{\rmS}(V),V)$$ and 
   $$ \mathfrak{C'}_{\RB}(V)=\Hom(\overline{\rmS}(V),V).$$
    Set $$\mathfrak{C'}_{\RBA}(V) =s\mathfrak{C'}_{\Lie}(V)\oplus\mathfrak{C'}_{\RB}(V).$$  It is not difficult to see that    $\mathfrak{C'}_{\RBA}(V)$ is  an $L_\infty$-algebra. The $L_\infty[1]$-structure is given by:
   
   \begin{thm}[Controlling algebra of homotopy version\cite{yj}]\label{thm: Linfinity for homotopy Lie}
   	Keep the above notation and denote $t=\sqrt{\mu^2-4\lambda}$, there exists an $L_\infty[1]$-algebra structure on  ${\mathfrak{C'}_{\RBA}}(V)$, where $l'_i$ are given by 
   	$$l'_1(sf)=-A_n f$$
   	$$
   	l'_2(sf,sg)      =    (-1)^{|f|} s[f,g]_{\NR}, $$
   	and for $i\geq 2$,
   	\begin{equation*}
   		\begin{aligned}
   			l'_i(sf,\xi_1,\cdots,\xi_{i-1})  =&\sum\limits_{\tau\in\Sh(m_1+1,\dots,m_{i-1}+1,n+2-i)}sgn(\tau)\\ &
   			-A_{n-i+1}\big(f(\xi_1\otimes\dots\otimes\xi_{i-1}\otimes\Id^{\otimes n+2-i})\big)\tau^{-1}\\
   			&+\sum_{k=1}^{i-1}\sum\limits_{\tau\in\Sh(m_1+1,\dots,m_{i-1}+1,n+3-i,m_k)}sgn(\tau)(-1)^{(|f|+\sum\limits_{j=1 }^{k-1}|\xi_j|)|\xi_k|} \\ &
   			-B_{n-i+3} \xi_k  \big(f(\xi_1\otimes\dots\otimes\xi_{i-1}\otimes\Id^{\otimes n+3-i})\o \Id^{\otimes m_k}\big)\tau^{-1}	
   		\end{aligned}
   	\end{equation*}
   	
   	For homogeneous elements	$f\in\Hom(\rmS^{n+1}(V),V)\subseteq \mathfrak{C'}_{\Lie}(V)$, $g\in\Hom(\rmS^{m+1}(V),V)\subseteq \mathfrak{C'}_{\Lie}(V)$    	and $\xi_j\in\Hom(\rmS^{m_j+1}(V),V)\subseteq \mathfrak{C'}_{\RB}(V)$, $1\leq j\leq i-1$, and all other components vanish.
   \end{thm}
   
   Similarly, we have a graded  version of Proposition \ref{L infty morphism}.
   \begin{prop}
   	With the same underlying graded space. There exists an $L_\infty[1]$-algebra morphism $\phi:({\mathfrak{C}_{\RBA}}(A),\{l_i\})\to ({\mathfrak{C'}_{\RBA}}(A),\{l'_i\}) $  where $\phi$ is given by:\\
   	$sf \mapsto \sum_{\sigma \in \rmS_n} \epsilon(\sigma) sf \sigma^{-1},$ \\
   	$g\mapsto \sum_{\sigma \in \rmS_m} \epsilon(\sigma) g \sigma^{-1},$ \\
   	and $\phi_i$ vanish for $i\ge 2$,
   	
   	for homogeneous elements	$f\in\Hom(A^{\otimes n},A)\subseteq \frakm$, $g\in\Hom(A^{\otimes m},A)\subseteq \fraka$.
   \end{prop}

   We know that $(A,\pi,T)$ is a homotopy  extended Rota-Baxter algebra if and only if $(\pi=\{\pi_i\}_{i\geqslant 1},T=\{T_i\}_{i\geqslant 1})$, where $\pi_i,T_i\in \Hom(A^{\otimes i},A)$ for each $i$, is a Maurer-Cartan element of $({\mathfrak{C}_{\RBA}}(V),\{l_i\})$   by Proposition-Definition \ref{Def: homotopy RB algebras}. By Proposition \ref{relations},  we have:
  
  \begin{prop}
  	Let  $(A,\pi=\{\pi_i\}_{i\geqslant 1},T=\{T_i\}_{i\geqslant 1})$ be a homotopy  extended Rota-Baxter algebra. Then it  induces a homotopy extended Rota-Baxter Lie algebra on $A$ where the $L_\infty$-structure $m=\{m_i\}_{i\geqslant 1}$ is given by $m_n:=\sum_{\sigma \in \rmS_n} \epsilon(\sigma) \pi_n \sigma$, and the homotopy  extended Rota-Baxter operator $T'=\{T'_i\}_{i\geqslant 1}$ on the $L_\infty$-algebra $(A,m=\{m_i\}_{i\geqslant 1})$ is given by $T'_n:=\sum_{\sigma \in \rmS_n} \epsilon(\sigma) T_n \sigma$.
  \end{prop}
    
    \bigskip
    
    \section{Appendix:The proof of Theorem \ref{thm: Linfinity for ass case}}\label{appendix}

	We just need to justify the equalities 
	\begin{align*}
		\sum_{i=1}^{n+1}\sum_{\sigma\in \Sh(i,n+1-i)}\epsilon(\sigma)l_{n-i+2}\circ(l_i\otimes \Id^{\otimes n-i+1})\sigma^{-1}(sf,sg,\xi_1,\dots,\xi_{n-1})=0,
	\end{align*}
    for $|f|=l,|g|=m,|\xi_j|=m_j, j\ge 1.$ All other cases are trivial.
	
	\begin{align*}
			LHS=&	l_{n+1}(l_1(sf),sg,\xi_1,\dots,\xi_{n-1})+(-1)^{(l+1)(m+1)}l_{n+1}(l_1(sg),sf,\xi_1,\dots,\xi_{n-1})\qquad(\text{1})\\ 
			&+l_{n}(l_2(sf,sg),\xi_1,\dots,\xi_{n-1})\qquad(\text{2})\\  
			&+\sum_{i=2}^{n}\sum_{\sigma\in \Sh(i-1,n-i)} \epsilon(\sigma)(-1)^{l(m+1)} l_{n-i+2}(sg,l_i(sf,\xi_{\sigma(1)},\dots,\xi_{\sigma(i-1)}),\xi_{\sigma(i)},\dots,\xi_{\sigma(n-1)})\qquad(\text{3})\\ 
			&+\sum_{i=2}^{n}\sum_{\sigma\in \Sh(i-1,n-i)} \epsilon(\sigma)(-1)^{l+1}l_{n-i+2}(sf,l_i(sg,\xi_{\sigma(1)},\dots,\xi_{\sigma(i-1)}),\xi_{\sigma(i)},\dots,\xi_{\sigma(n-1)})\qquad(\text{4})     
       \end{align*}

     \begin{align*}
     		(1)=&-A_{l}(-1)^{l(m+1)}l_{n+1}(sg,\xi_0,\xi_1,\dots,\xi_{n-1})\\
     		 &\quad(\text{let } \xi^{1'}_{0}:=f \text{ and all other } \xi^{1'}_{k}:=\xi_{k}) \qquad(\text{1.1})\\  
     		-&A_{m}(-1)^{l+1}l_{n+1}(sf,\xi_{0},\xi_1,\dots,\xi_{n-1})\\
     		& \quad(\text{let } \xi^{1''}_{0}:=g \text{ and all other } \xi^{1''}_{k}:=\xi_{k}) \qquad(\text{1.2})\\   
     		=&A_l (-1)^{l(m+1)}\sum\limits_{\substack{\sigma\in \rmS_{n}\\a_j+m_{\sigma(j)}+1\le a_{j+1}\\j\ge 0}}\epsilon(\sigma)\\ 
     		&A_{m-n} \big(\cdots (((g\bar{\circ}_{a_0}\xi^{1'}_{\sigma(0)})\bar{\circ}_{a_1}\xi^{1'}_{\sigma(1)})\bar{\circ}_{a_2}\cdots)\bar{\circ}_{a_{n-1}}\xi^{1'}_{\sigma(n-1)}   \qquad(\text{1.1.1})\\
     		&+A_{l}(-1)^{l(m+1)}\sum\limits_{\substack{\sigma\in \rmS_{n}\\a_j+m_{\sigma(j)}+1\le a_{j+1}\\ j\ge 1}} \epsilon(\sigma) (-1)^{mm^{1'}_{\sigma(0)}} \\ 
     		&B_{m-n+2} \xi^{1'}_{\sigma(0)}\bar{\circ}\big(\cdots ((g\bar{\circ}_{a_1}\xi^{1'}_{\sigma(1)})\bar{\circ}_{a_3}\cdots)\bar{\circ}_{a_{n-1}}\xi^{1'}_{\sigma(n-1)})  \qquad(\text{1.1.2})\\
     		&+A_{m}(-1)^{l+1}\sum\limits_{\substack{\sigma\in \rmS_{n}\\a_j+m_{\sigma(j)+1}\le a_{j+1}\\j\ge 0}}\epsilon(\sigma)\\ 
     		&A_{l-n} \big(\cdots (((f\bar{\circ}_{a_0}\xi^{1''}_{\sigma(0)})\bar{\circ}_{a_1}\xi^{1''}_{\sigma(1)})\bar{\circ}_{a_2}\cdots)\bar{\circ}_{a_{n-1}}\xi^{1''}_{\sigma(n-1)}  \qquad(\text{1.2.1})\\
     		&+A_{m}(-1)^{l+1}\sum\limits_{\substack{\sigma\in \rmS_{n}\\a_j+m_{\sigma(j)}+1\le a_{j+1}\\ j\ge 1}} \epsilon(\sigma) (-1)^{lm^{1''}_{\sigma(0)}}\\ 
     		&B_{l-n+2} \xi^{1''}_{\sigma(0)}\bar{\circ}\big(\cdots (((f\bar{\circ}_{a_1}\xi^{1''}_{\sigma(1)})\bar{\circ}_{a_2}\xi^{1''}_{\sigma(2)})\bar{\circ}_{a_3}\cdots)\bar{\circ}_{a_{n-1}}\xi^{1''}_{\sigma(n-1)})  \qquad(\text{1.2.2})
     \end{align*}

    \begin{align*}
    		(1.1.2)&=A_{l}(-1)^{l(m+1)}\sum\limits_{\substack{\sigma\in \rmS_{n}\\a_j+m_{\sigma(j)}+1\le a_{j+1}\\j\ge 1}} \epsilon(\sigma) (-1)^{mm^{1'}_{\sigma(0)}} \\ 
    		&B_{m-n+2} \xi^{1'}_{\sigma(0)}\bar{\circ}\big(\cdots ((g\bar{\circ}_{a_1}\xi^{1'}_{\sigma(1)})\bar{\circ}_{a_3}\cdots)\bar{\circ}_{a_{n-1}}\xi^{1'}_{\sigma(n-1)})\text{ with $\sigma(0)=0$}\qquad(\text{1.1.2.1})\\
    		+& \text{others}\qquad(\text{1.1.2.2})
    \end{align*}

   \begin{align*}
   		(1.2.2)&=A_{m}(-1)^{l+1}\sum\limits_{\substack{\sigma\in \rmS_{n}\\a_j+m_{\sigma(j)}+1\le a_{j+1}\\ j\ge 1}} \epsilon(\sigma) (-1)^{lm^{1''}_{\sigma(0)}}\\ 
   		&B_{l-n+2} \xi^{1''}_{\sigma(0)}\bar{\circ}\big(\cdots (((f\bar{\circ}_{a_1}\xi^{1''}_{\sigma(1)})\bar{\circ}_{a_2}\xi^{1''}_{\sigma(2)})\bar{\circ}_{a_3}\cdots)\bar{\circ}_{a_{n-1}}\xi^{1''}_{\sigma(n-1)})\text{ with $\sigma(0)=0$}\qquad(\text{1.2.2.1})\\
   		+& \text{others}\qquad(\text{1.2.2.2})
   \end{align*}

	\begin{align*}
			(2)=&(-1)^l l_{n}(s[f,g],\xi_1,\dots,\xi_{n-1})\qquad(\text{2})\\
			=&(-1)^{l+1}
			\sum\limits_{\substack{\sigma\in \rmS_{n-1}\\a_j+m_{\sigma(j)+1}\le a_{j+1}\\j\ge 1}}	A_{m+l-n+1}\epsilon(\sigma)\\ 
			&\big(\cdots((f\bar{\circ}g)\bar{\circ}_{a_1}\xi_{\sigma(1)})\bar{\circ}_{a_2}\cdots)\bar{\circ}_{a_{n-1}}\xi_{\sigma(n-1)})  \qquad(\text{2.1.1})\\ 
			&+(-1)^{l+1}
			\sum\limits_{\substack{\sigma\in \rmS_{n-1}\\a_j+m_{\sigma(j)+1}\le a_{j+1}\\j\ge 2}}	B_{m+l-n+3}\epsilon(\sigma)(-1)^{(l+m)m_{\sigma(1)}} \\ 
			&\xi_{\sigma(1)}\bar{\circ}\big(\cdots((f\bar{\circ}g)\bar{\circ}_{a_2}\xi_{\sigma(2)})\bar{\circ}_{a_3}\cdots)\bar{\circ}_{a_{n-1}}\xi_{\sigma(n-1)}   \qquad(\text{2.1.2})\\	
			&+(-1)^{l+ml}
			\sum\limits_{\substack{\sigma\in \rmS_{n-1}\\a_j+m_{\sigma(j)+1\le a_{j+1}}\\j\ge 1}}	A_{m+l-n+1}\epsilon(\sigma)\\ 
			&\big(\cdots((g\bar{\circ}f)\bar{\circ}_{a_1}\xi_{\sigma(1)})\bar{\circ}_{a_2}\cdots)\bar{\circ}_{a_{n-1}}\xi_{\sigma(n-1)}) \qquad(\text{2.2.1})\\
			&+(-1)^{l+ml}
			\sum\limits_{\substack{\sigma\in \rmS_{n-1}\\a_j+m_{\sigma(j)+1\le a_{j+1}}\\j\ge 2}}	B_{m+l-n+3}\epsilon(\sigma)(-1)^{(l+m)m_{\sigma(1)}} \\ 
			&\xi_{\sigma(1)}\bar{\circ}\big(\cdots((g\bar{\circ}f)\bar{\circ}_{a_2}\xi_{\sigma(2)})\bar{\circ}_{a_3}\cdots)\bar{\circ}_{a_{n-1}}\xi_{\sigma(n-1)}) \qquad(\text{2.2.2})
    \end{align*}

	\begin{align*}
			(3)=&\sum_{i=2}^{n}\sum_{\sigma\in \Sh(i-1,n-i)} \epsilon(\sigma)(-1)^{l(m+1)}	l_{n-i+2}(sg, \sum\limits_{\substack{\sigma_1\in \rmS_{i-1}\\a_j+m_{\sigma\sigma_1(j)}+1\le a_{j+1}\\j\ge 1}}\epsilon(\sigma_1) \\ 
			&-A_{l-i+1}  \xi_{\sigma(i-1)},\xi_{\sigma(i)},\dots,\xi_{\sigma(n-1)})\qquad(\text{3.1})\\     
			&(\text{Let}\quad \xi^{3'}_{\sigma(i-1)}:=\big(\cdots ((f\bar{\circ}_{a_1}\xi_{\sigma\sigma_1(1)})\bar{\circ}_{a_2}\xi_{\sigma\sigma_1(2)})\bar{\circ}_{a_3}\cdots)\bar{\circ}_{a_{i-1}}\xi_{\sigma\sigma_1(i-1)} \text{ and all other } \xi^{3'}_{k}:=\xi_{k})\\
			&+\sum_{i=2}^{n}\sum_{\sigma\in \Sh(i-1,n-i)} \epsilon(\sigma)(-1)^{l(m+1)}	l_{n-i+2}(sg, \sum\limits_{\substack{\sigma_1\in \rmS_{i-1}\\a_j+m_{\sigma\sigma_1(j)}+1\le a_{j+1}\\ j\ge 2}} \epsilon(\sigma_1) (-1)^{lm_{\sigma\sigma_1(1)}} \\ 
			&-B_{l-i+3} \xi_{\sigma(i-1)},\xi_{\sigma(i)},\dots,\xi_{\sigma(n-1)})\qquad(\text{3.2})\\
			&(\text{Let}\quad \xi^{3''}_{\sigma(i-1)}:=\xi_{\sigma\sigma_1(1)}\bar{\circ}\big(\cdots ((f\bar{\circ}_{a_2}\xi_{\sigma\sigma_1(2)})\bar{\circ}_{a_3}\cdots)\bar{\circ}_{a_{i-1}}\xi_{\sigma\sigma_1(i-1)}) \text{ and all other } \xi^{3''}_{k}:=\xi_{k})\\  
			=&\sum_{i=2}^{n}\sum_{\sigma\in \Sh(i-1,n-i)} \epsilon(\sigma)(-1)^{l(m+1)} \sum\limits_{\substack{\sigma_1\in \rmS_{i-1}\\a_j+m_{\sigma\sigma_1(j)}+1\le a_{j+1}\\ j\ge 1}}\epsilon(\sigma_1)\sum\limits_{\substack{\sigma_2\in \rmS_{n-i+1}\\a^{3'}_{j'}+m^{3'}_{\sigma\sigma_2(j')}+1\le a^{3'}_{j'+1}\\j'\ge i-1}}\epsilon(\sigma_2) \\ 
			&A_{l-i+1} A_{m-n+i-1} (((g\bar{\circ}_{a^{3'}_{i-1}}\xi^{3'}_{\sigma\sigma_2(i-1)})\bar{\circ}_{a^{3'}_i}\xi^{3'}_{\sigma\sigma_2(i)})\bar{\circ}_{a^{3'}_{i+1}}\cdots)\bar{\circ}_{a^{3'}_{n-1}}\xi^{3'}_{\sigma\sigma_2(n-1)} \qquad(\text{3.1.1})\\
			&+\sum_{i=2}^{n}\sum_{\sigma\in \Sh(i-1,n-i)} \epsilon(\sigma)(-1)^{l(m+1)} \sum\limits_{\substack{\sigma_1\in \rmS_{i-1}\\a_j+m_{\sigma\sigma_1(j)}+1\le a_{j+1}\\ j\ge 1}}\epsilon(\sigma_1) \sum\limits_{\substack{\sigma_2\in \rmS_{n-i+1}\\a^{3'}_{j'}+m^{3'}_{\sigma\sigma_2(j')}+1\le a^{3'}_{j'+1}\\j'\ge i}}            \epsilon(\sigma_2)(-1)^{m m^{3'}_{\sigma\sigma_2(i-1)}}\\ 
			&A_{l-i+1} B_{m-n+i+1} \xi^{3'}_{\sigma\sigma_2(i-1)}\bar{\circ}(\cdots(g\bar{\circ}_{a^{3'}_{i}}\xi^{3'}_{\sigma\sigma_2(i)})\bar{\circ}_{a^{3'}_{i+1}}\cdots)\bar{\circ}_{a^{3'}_{n-1}}\xi^{3'}_{\sigma\sigma_2(n-1)} \qquad(\text{3.1.2})\\
			&+\sum_{i=2}^{n}\sum_{\sigma\in \Sh(i-1,n-i)} \epsilon(\sigma)(-1)^{l(m+1)} \sum\limits_{\substack{\sigma_1\in \rmS_{i-1}\\a_j+m_{\sigma\sigma_1(j)}+1\le a_{j+1}\\j\ge 2}}\epsilon(\sigma_1)(-1)^{lm_{\sigma\sigma_1(1)}}\sum\limits_{\substack{\sigma_2\in \rmS_{n-i+1}\\a^{3''}_{j'}+m^{3''}_{\sigma\sigma_2(j')}+1\le a^{3''}_{j'+1}\\j' \ge i-1}} \epsilon(\sigma_2) \\ 
			&B_{l-i+3} A_{m-n+i-1} (((g\bar{\circ}_{a^{3''}_{i-1}}\xi^{3''}_{\sigma\sigma_2(i-1)}\bar{\circ}_{a^{3''}_i}\xi^{3''}_{\sigma\sigma_2(i)})\bar{\circ}_{a^{3''}_{i+1}}\cdots)\bar{\circ}_{a^{3''}_{n-1}}\xi^{3''}_{\sigma\sigma_2(n-1)} \qquad(\text{3.2.1})\\
			&+\sum_{i=2}^{n}\sum_{\sigma\in \Sh(i-1,n-i)} \epsilon(\sigma)(-1)^{l(m+1)} \sum\limits_{\substack{\sigma_1\in \rmS_{i-1}\\a_j+m_{\sigma\sigma_1(j)}+1\le a_{j+1}\\j\ge 2}}\epsilon(\sigma_1)(-1)^{lm_{\sigma\sigma_1(1)}}\sum\limits_{\substack{\sigma_2\in \rmS_{n-i+1}\\a^{3''}_{j'}+m^{3''}_{\sigma\sigma_2(j')}+1\le a^{3''}_{j'+1}\\j' \ge i}} \epsilon(\sigma_2) \\ 
			&(-1)^{mm^{3''}_{\sigma\sigma_2(i-1)}}B_{l-i+3} B_{m-n+i+1} \xi^{3''}_{\sigma\sigma_2(i-1)}\bar{\circ}(\cdots(g\bar{\circ}_{a^{3''}_{i}}\xi^{3''}_{\sigma\sigma_2(i)})\bar{\circ}_{a^{3''}_{i+1}}\cdots)\bar{\circ}_{a^{3''}_{n-1}}\xi^{3''}_{\sigma\sigma_2(n-1)}\qquad(\text{3.2.2})	
	\end{align*}

	\begin{align*}
			(3.1.2)&=\sum_{i=2}^{n}\sum_{\sigma\in \Sh(i-1,n-i)} \epsilon(\sigma)(-1)^{l(m+1)} \sum\limits_{\substack{\sigma_1\in \rmS_{i-1}\\a_j+m_{\sigma\sigma_1(j)}+1\le a_{j+1}\\j\ge 1}}\epsilon(\sigma_1) \sum\limits_{\substack{\sigma_2\in \rmS_{n-i+1}\\a^{3'}_{j'}+m^{3'}_{\sigma\sigma_2(j')}+1\le a^{3'}_{j'+1}\\j'\ge i}}\epsilon(\sigma_2)\\ 
			&(-1) ^{m m^{3'}_{\sigma\sigma_2(i-1)}}A_{l-i+1} B_{m-n+i+1} \xi^{3'}_{\sigma\sigma_2(i-1)}\bar{\circ}(\cdots(g\bar{\circ}_{a^{3'}_{i}}\xi^{3'}_{\sigma\sigma_2(i)})\bar{\circ}_{a^{3'}_{i+1}}\cdots)\bar{\circ}_{a^{3'}_{n-1}}\xi^{3'}_{\sigma\sigma_2(n-1)}\\ 
			&\text{ with $\sigma_2(i-1)\neq i-1$}\qquad(\text{3.1.2.1})\\
			+&\sum_{i=2}^{n}\sum_{\sigma\in \Sh(i-1,n-i)} \epsilon(\sigma)(-1)^{l(m+1)} \sum\limits_{\substack{\sigma_1\in \rmS_{i-1}\\a_j+m_{\sigma\sigma_1(j)}+1\le a_{j+1}\\j\ge 1}}\epsilon(\sigma_1) \sum\limits_{\substack{\sigma_2\in \rmS_{n-i+1}\\a^{3'}_{j'}+m^{3'}_{\sigma\sigma_2(j')}+1\le a^{3'}_{j'+1}\\j'\ge i}}\epsilon(\sigma_2)\\ 
			&(-1) ^{m m^{3'}_{\sigma\sigma_2(i-1)}}A_{l-i+1} B_{m-n+i+1} \xi^{3'}_{\sigma\sigma_2(i-1)}\bar{\circ}(\cdots(g\bar{\circ}_{a^{3'}_{i}}\xi^{3'}_{\sigma\sigma_2(i)})\bar{\circ}_{a^{3'}_{i+1}}\cdots)\bar{\circ}_{a^{3'}_{n-1}}\xi^{3'}_{\sigma\sigma_2(n-1)}\\ 
			&\text{ with $\sigma_2(i-1)=i-1$,  $f$ and $g$ are directly composed}\qquad(\text{3.1.2.2})\\
			+& \text{others}\qquad(\text{3.1.2.3})
	\end{align*}
    
    \begin{rmk}
    	$f$ and $g$ are directly composed means that $f \bar{\circ}_i g$ for some $i$ appears in this iterative composition.
    \end{rmk}

   \begin{align*}
		(3.2.2)&=\sum_{i=2}^{n}\sum_{\sigma\in \Sh(i-1,n-i)} \epsilon(\sigma)(-1)^{l(m+1)} \sum\limits_{\substack{\sigma_1\in \rmS_{i-1}\\a_j+m_{\sigma\sigma_1(j)}+1\le a_{j+1}\\j\ge 2}}\epsilon(\sigma_1)(-1)^{lm_{\sigma\sigma_1(1)}}\sum\limits_{\substack{\sigma_2\in \rmS_{n-i+1}\\a^{3''}_{j'}+m^{3''}_{\sigma\sigma_2(j')}+1\le a^{3''}_{j'+1}\\j'\ge i}}\epsilon(\sigma_2) \\ 
		&(-1)^{mm^{3''}_{\sigma\sigma_2(i-1)}}B_{l-i+3} B_{m-n+i+1} \xi^{3''}_{\sigma\sigma_2(i-1)}\bar{\circ}(\cdots(g\bar{\circ}^{3''}_{a_{i}}\xi^{3''}_{\sigma\sigma_2(i)})\bar{\circ}_{a^{3''}_{i+1}}\cdots)\bar{\circ}_{a^{3''}_{n-1}}\xi^{3''}_{\sigma\sigma_2(n-1)}\\
		&\text{ with $\sigma_2(i-1)=i-1$,  $f$ and $g$ are directly composed}\qquad(\text{3.2.2.1})\\
		+& \text{others}\qquad(\text{3.2.2.2})
   \end{align*}

	\begin{align*}
			(4)=&\sum_{i=2}^{n}\sum_{\sigma\in \Sh(i-1,n-i)} \epsilon(\sigma)(-1)^{l+1}l_{n-i+2}(sf,\sum\limits_{\substack{\sigma_1\in \rmS_{i-1}\\a_j+m_{\sigma\sigma_1(j)}+1\le a_{j+1}\\j\ge 1}}\epsilon(\sigma_1) \\ 
			&-A_{m-i+1}  \big(\cdots ((g\bar{\circ}_{a_1}\xi_{\sigma\sigma_1(1)})\bar{\circ}_{a_2}\xi_{\sigma\sigma_1(2)})\bar{\circ}_{a_3}\cdots)\bar{\circ}_{a_{i-1}}\xi_{\sigma\sigma_1(i-1)},\xi_{\sigma(i)},\dots,\xi_{\sigma(n-1)})\qquad(\text{4.1})\\     
			&(\text{Let}\quad \xi^{4'}_{\sigma(i-1)}:=\big(\cdots ((g\bar{\circ}_{a_1}\xi_{\sigma\sigma_1(1)})\bar{\circ}_{a_2}\xi_{\sigma\sigma_1(2)})\bar{\circ}_{a_3}\cdots)\bar{\circ}_{a_{i-1}}\xi_{\sigma\sigma_1(i-1)}\text{ and all other } \xi^{4'}_{k}:=\xi_{k})\\
			&+\sum_{i=2}^{n} \sum_{\sigma\in \Sh(i-1,n-i)} \epsilon(\sigma)(-1)^{l+1} l_{n-i+2}(sf,\sum\limits_{\substack{\sigma\in \rmS_{i-1}\\a_j+m_{\sigma(j)}+1\le a_{j+1}\\ j\ge 2}} \epsilon(\sigma_1) (-1)^{mm_{\sigma\sigma_1(1)}} \\ 
			&-B_{m-i+3}\xi_{\sigma\sigma_1(1)}\bar{\circ}_{a_1}\big(\cdots ((g\bar{\circ}_{a_2}\xi_{\sigma\sigma_1(2)})\bar{\circ}_{a_3}\cdots)\bar{\circ}_{a_{i-1}}\xi_{\sigma\sigma_1(i-1)}),\xi_{\sigma(i)},\dots,\xi_{\sigma(n-1)})\qquad(\text{4.2})\\  
			&(\text{Let}\quad \xi^{4''}_{\sigma(i-1)}:=\xi_{\sigma\sigma_1(1)}\bar{\circ}\big(\cdots ((g\bar{\circ}_{a_2}\xi_{\sigma\sigma_1(2)})\bar{\circ}_{a_3}\cdots)\bar{\circ}_{a_{i-1}}\xi_{\sigma\sigma_1(i-1)}\text{ and all other } \xi^{4''}_{k}:=\xi_{k})\\
			=&\sum_{i=2}^{n}\sum_{\sigma\in \Sh(i-1,n-i)} \epsilon(\sigma)(-1)^{l+1} \sum\limits_{\substack{\sigma_1\in \rmS_{i-1}\\a_j+m_{\sigma\sigma_1(j)}+1\le a_{j+1}\\j\ge 1}}\epsilon(\sigma_1)\sum\limits_{\substack{\sigma_2\in \rmS_{n-i+1}\\a^{4'}_{j'}+m^{4'}_{\sigma\sigma_2(j')}+1\le a^{4'}_{j'+1}\\j' \ge i-1}}\epsilon(\sigma_2)  \\ 
			&A_{m-i+1} A_{l-n+i-1} (((f\bar{\circ}_{a^{4'}_{i-1}}\xi^{4'}_{\sigma\sigma_2(i-1)}\bar{\circ}_{a^{4'}_i}\xi^{4'}_{\sigma\sigma_2(i)})\bar{\circ}_{a^{4'}_{i+1}}\cdots)\bar{\circ}_{a^{4'}_{n-1}}\xi^{4'}_{\sigma\sigma_2(n-1)} \qquad(\text{4.1.1})\\
			&+\sum_{i=2}^{n}\sum_{\sigma\in \Sh(i-1,n-i)} \epsilon(\sigma)(-1)^{l+1} \sum\limits_{\substack{\sigma_1\in \rmS_{i-1}\\a_j+m_{\sigma\sigma_1(j)}+1\le a_{j+1}\\j\ge 1}}\epsilon(\sigma_1)\sum\limits_{\substack{\sigma_2\in \rmS_{n-i+1}\\a^{4'}_{j'}+m^{4'}_{\sigma\sigma_2(j')}+1\le a^{4'}_{j'+1}\\ j'\ge i}}\epsilon(\sigma_2) (-1)^{lm^{4'}_{\sigma\sigma_2(i-1)}}\\ 
			&A_{m-i+1} B_{l-n+i+1} \xi^{4'}_{\sigma\sigma_2(i-1)}\bar{\circ}(\cdots(f\bar{\circ}_{a^{4'}_{i}}\xi^{4'}_{\sigma\sigma_2(i)})\bar{\circ}_{a^{4'}_{i+1}}\cdots)\bar{\circ}_{a^{4'}_{n-1}}\xi^{4'}_{\sigma\sigma_2(n-1)} \qquad(\text{4.1.2})\\
			&+\sum_{i=2}^{n}\sum_{\sigma\in \Sh(i-1,n-i)} \epsilon(\sigma)(-1)^{l+1} \sum\limits_{\substack{\sigma_1\in \rmS_{i-1}\\a_j+m_{\sigma\sigma_1(j)}+1\le a_{j+1}\\j\ge 2}}\epsilon(\sigma_1)(-1)^{mm_{\sigma\sigma_1(1)}}\sum\limits_{\substack{\sigma_2\in \rmS_{n-i+1}\\a^{4''}_{j'}+m^{4''}_{\sigma\sigma_2(j')}+1\le a^{4''}_{j'+1}\\j' \ge i-1}}\epsilon(\sigma_2) \\ 
			&B_{m-i+3} A_{l-n+i-1} (((f\bar{\circ}_{a^{4''}_{i-1}}\xi^{4''}_{\sigma\sigma_2(i-1)}\bar{\circ}_{a^{4''}_i}\xi^{4''}_{\sigma\sigma_2(i)})\bar{\circ}_{a^{4''}_{i+1}}\cdots)\bar{\circ}_{a^{4''}_{n-1}}\xi^{4''}_{\sigma\sigma_2(n-1)} \qquad(\text{4.2.1})\\
			&+\sum_{i=2}^{n}\sum_{\sigma\in \Sh(i-1,n-i)} \epsilon(\sigma)(-1)^{l+1} \sum\limits_{\substack{\sigma_1\in \rmS_{i-1}\\a_j+m_{\sigma\sigma_1(j)}+1\le a_{j+1}\\j\ge 2}}\epsilon(\sigma_1)(-1)^{mm_{\sigma\sigma_1(1)}}\sum\limits_{\substack{\sigma_2\in \rmS_{n-i+1}\\a^{4''}_{j'}+m^{4''}_{\sigma\sigma_2(j')}+1\le a^{4''}_{j'+1}\\j'\ge i}}\epsilon(\sigma_2) \\ 
			&(-1)^{lm^{4''}_{\sigma\sigma_2(i-1)}}B_{m-i+3} B_{l-n+i+1} \xi^{4''}_{\sigma\sigma_2(i-1)}\bar{\circ}(\cdots(f\bar{\circ}_{a^{4''}_{i}}\xi^{4''}_{\sigma\sigma_2(i)})\bar{\circ}_{a^{4''}_{i+1}}\cdots)\bar{\circ}_{a^{4''}_{n-1}}\xi^{4''}_{\sigma\sigma_2(n-1)} \qquad(\text{4.2.2})
	\end{align*}

    \begin{align*}
    		(4.1.2)&=\sum_{i=2}^{n}\sum_{\sigma\in \Sh(i-1,n-i)} \epsilon(\sigma)(-1)^{l+1} \sum\limits_{\substack{\sigma_1\in \rmS_{i-1}\\a_j+m_{\sigma\sigma_1(j)}+1\le a_{j+1}\\j\ge 1}}\epsilon(\sigma_1)\sum\limits_{\substack{\sigma_2\in \rmS_{n-i+1}\\a^{4'}_{j'}+m^{4'}_{\sigma\sigma_2(j')}+1\le a^{4'}_{j'+1}\\ j'\ge i}}\epsilon(\sigma_2) (-1)^{lm^{4'}_{\sigma\sigma_2(i-1)}}\\ 
    		&A_{m-i+1} B_{l-n+i+1} \xi^{4'}_{\sigma\sigma_2(i-1)}\bar{\circ}(\cdots(f\bar{\circ}_{a^{4'}_{i}}\xi^{4'}_{\sigma\sigma_2(i)})\bar{\circ}_{a^{4'}_{i+1}}\cdots)\bar{\circ}_{a^{4'}_{n-1}}\xi^{4'}_{\sigma\sigma_2(n-1)}\\
    		&\text{ with $\sigma_2(i-1)\neq i-1$}\qquad(\text{4.1.2.1})\\
    		+&\sum_{i=2}^{n}\sum_{\sigma\in \Sh(i-1,n-i)} \epsilon(\sigma)(-1)^{l+1} \sum\limits_{\substack{\sigma_1\in \rmS_{i-1}\\a_j+m_{\sigma\sigma_1(j)}+1\le a_{j+1}\\j\ge 1}}\epsilon(\sigma_1)\sum\limits_{\substack{\sigma_2\in \rmS_{n-i+1}\\a^{4'}_{j'}+m^{4'}_{\sigma\sigma_2(j')}+1\le a^{4'}_{j'+1}\\ j'\ge i}}\epsilon(\sigma_2) (-1)^{lm^{4'}_{\sigma\sigma_2(i-1)}}\\ 
    		&A_{m-i+1} B_{l-n+i+1} \xi^{4'}_{\sigma\sigma_2(i-1)}\bar{\circ}(\cdots(f\bar{\circ}_{a^{4'}_{i}}\xi^{4'}_{\sigma\sigma_2(i)})\bar{\circ}_{a^{4'}_{i+1}}\cdots)\bar{\circ}_{a^{4'}_{n-1}}\xi^{4'}_{\sigma\sigma_2(n-1)}\\
    		&\text{ with $\sigma_2(i-1)=i-1$,  $g$ and $f$ are directly composed}\qquad(\text{4.1.2.2})\\
    		+& \text{others}\qquad(\text{4.1.2.3})
    \end{align*}

    \begin{align*}
    		(4.2.2)=&\sum_{i=2}^{n}\sum_{\sigma\in \Sh(i-1,n-i)} \epsilon(\sigma)(-1)^{l+1} \sum\limits_{\substack{\sigma_1\in \rmS_{i-1}\\a_j+m_{\sigma\sigma_1(j)}+1\le a_{j+1}\\j\ge 2}}\epsilon(\sigma_1)(-1)^{mm_{\sigma\sigma_1(1)}}\sum\limits_{\substack{\sigma_2\in \rmS_{n-i+1}\\a^{4''}_{j'}+m^{4''}_{\sigma\sigma_2(j')}+1\le a^{4''}_{j'+1}\\j'\ge i}}\epsilon(\sigma_2) \\ 
    		&(-1)^{lm^{4''}_{\sigma\sigma_2(i-1)}}B_{m-i+3} B_{l-n+i+1} \xi^{4''}_{\sigma\sigma_2(i-1)}\bar{\circ}(\cdots(f\bar{\circ}_{a^{4''}_{i}}\xi^{4''}_{\sigma\sigma_2(i)})\bar{\circ}_{a^{4''}_{i+1}}\cdots)\bar{\circ}_{a^{4''}_{n-1}}\xi^{4''}_{\sigma\sigma_2(n-1)}\\
    		&\text{ with $\sigma_2(i-1)= i-1$,  $g$ and $f$ are directly composed}\qquad(\text{4.2.2.1})\\
    		+& \text{others}\qquad(\text{4.2.2.2})
    \end{align*}

   With the same form of operation of iterative composition, using the Koszul sign and Lemma \ref{key lemma}, we have the following identities,
   \begin{align*}
   	&(1.1.2.1)+(1.2.1)+(2.1.1)+(3.1.2.2)+(4.1.1)=0,\\
   	&(1.2.2.2)+(2.1.2)+(3.2.2.1)+(4.1.2.1)=0,\\
   	&(1.1.1)+(1.2.2.1)+(2.2.1)+(3.1.1)+(4.1.2.2)=0,\\
   	&(1.1.2.2)+(2.2.2)+(3.1.2.1)+(4.2.2.1)=0,\\
   	&(3.1.2.3)+(4.2.1)=0,\\
   	&(3.2.1)+(4.1.2.3)=0,\\
   	&(3.2.2.2)+(4.2.2.2)=0.
   \end{align*}

   For example $(1.1.2.1)+(1.2.1)+(2.1.1)+(3.1.2.2)+(4.1.1)=0$.
   
   \begin{align*}
   LHS=&(-1)^{l}\sum\limits_{\substack{\sigma\in \rmS_{n}-1\\a_j+m_{\sigma(j)}+1\le a_{j+1}\\j\ge 1}} \epsilon(\sigma) (-1)^{mm^{1'}_{\sigma(0)}+ml} \\ 
   &A_{l}B_{m-n+2} \xi^{1'}_{\sigma(0)}\bar{\circ}\big(\cdots ((g\bar{\circ}_{a_1}\xi^{1'}_{\sigma(1)})\bar{\circ}_{a_3}\cdots)\bar{\circ}_{a_{n-1}}\xi^{1'}_{\sigma(n-1)})\text{ with $\sigma(0)=0$}\\
   +&(-1)^{l}\sum_{i=2}^{n}\sum_{\sigma\in \Sh(i-1,n-i)}\epsilon(\sigma)  \sum\limits_{\substack{\sigma_1\in \rmS_{i-1}\\a_j+m_{\sigma\sigma_1(j)}+1\le a_{j+1}\\j\ge 1}}\epsilon(\sigma_1) \sum\limits_{\substack{\sigma_2\in \rmS_{n-i+1}\\a^{3'}_{j'}+m^{3'}_{\sigma\sigma_2(j')}+1\le a^{3'}_{j'+1}\\j'\ge i}}\epsilon(\sigma_2)\\ 
   &(-1) ^{m m^{3'}_{\sigma\sigma_2(i-1)}+ml}A_{l-i+1} B_{m-n+i+1} \xi^{3'}_{\sigma\sigma_2(i-1)}\bar{\circ}(\cdots(g\bar{\circ}_{a^{3'}_{i}}\xi^{3'}_{\sigma\sigma_2(i)})\bar{\circ}_{a^{3'}_{i+1}}\cdots)\bar{\circ}_{a^{3'}_{n-1}}\xi^{3'}_{\sigma\sigma_2(n-1)}\\ 
   &\text{ with $\sigma_2(i-1)=i-1$ and  composite of $f$ and $g$ directly} \\
   +&\sum\limits_{\substack{\sigma\in \rmS_{n}-1\\a_j+m_{\sigma(j)+1}\le a_{j+1}\\j\ge 0}}(-1)^{l+1}\epsilon(\sigma)\\ 
   &A_{m}A_{l-n} \big(\cdots (((f\bar{\circ}_{a_0}\xi^{1''}_{\sigma(0)})\bar{\circ}_{a_1}\xi^{1''}_{\sigma(1)})\bar{\circ}_{a_2}\cdots)\bar{\circ}_{a_{n-1}}\xi^{1''}_{\sigma(n-1)}\\
   +& 
  (-1)^{l+1} 
  \sum\limits_{\substack{\sigma\in \rmS_{n-1}\\a_j+m_{\sigma(j)+1\le a_{j+1}}\\j\ge 1}}	A_{m+l-n+1}\epsilon(\sigma)\\ 
  &\big(\cdots((f\bar{\circ}g)\bar{\circ}_{a_1}\xi_{\sigma(1)})\bar{\circ}_{a_2}\cdots)\bar{\circ}_{a_{n-1}}\xi_{\sigma(n-1)}) \\ 
   +&(-1)^{l+1}\sum_{i=2}^{n}\sum_{\sigma\in \Sh(i-1,n-i)}\epsilon(\sigma)  \sum\limits_{\substack{\sigma_1\in \rmS_{i-1}\\a_j+m_{\sigma\sigma_1(j)}+1\le a_{j+1}\\j\ge 1}}\epsilon(\sigma_1)\sum\limits_{\substack{\sigma_2\in \rmS_{n-i+1}\\a^{4'}_{j'}+m^{4'}_{\sigma\sigma_2(j')}+1\le a^{4'}_{j'+1}\\j' \ge i-1}}\epsilon(\sigma_2)  \\ 
   &A_{m-i+1} A_{l-n+i-1} (((f\bar{\circ}_{a^{4'}_{i-1}}\xi^{4'}_{\sigma\sigma_2(i-1)}\bar{\circ}_{a^{4'}_i}\xi^{4'}_{\sigma\sigma_2(i)})\bar{\circ}_{a^{4'}_{i+1}}\cdots)\bar{\circ}_{a^{4'}_{n-1}}\xi^{4'}_{\sigma\sigma_2(n-1)} \\
   =&0\text{ by Lemma \ref{key lemma}}.
   \end{align*}

\end{document}